\documentclass[a4paper, 11pt]{amsart}
\usepackage{enumitem, amsmath, amssymb, amsthm, xspace, caption, subcaption, comment, mathtools}
\usepackage[hidelinks]{hyperref}
\usepackage{multirow}
\usepackage{marvosym}
\usepackage{xcolor}
\usepackage[normalem]{ulem} 
\usepackage{fullpage}

\usepackage{xr}
\newtheorem{thm}{Theorem}[section]
\newtheorem{lem}[thm]{Lemma}
\newtheorem{prop}[thm]{Proposition} 

\newtheorem{conj}[thm]{Conjecture}
\newtheorem{qu}[thm]{Question}

\newtheorem{claim}[thm]{Claim}
\newtheorem{fact}[thm]{Fact}

\newtheorem{remark}[thm]{Remark}

\newtheorem{defi}[thm]{Definition}
\theoremstyle{definition}
\newtheorem*{claim*}{Claim}
\theoremstyle{plain}

\newenvironment{romenumerate}[1][-0.25em]{%
\vspace{#1}
\begin{enumerate}%
\itemsep0pt \parskip0pt \parsep0pt%
 }
 {\end{enumerate} }

\DeclareMathOperator*{\esssup}{ess\,sup}
\DeclareMathOperator*{\essinf}{ess\,inf}
\DeclareMathOperator{\limcut}{{\square}-{\lim}}
\DeclareMathOperator{\liminfty}{{\infty}-{\lim}}
\DeclareMathOperator{\var}{Var} %
\DeclareMathOperator{\ber}{Ber} %

\makeatletter

\newif\ifdraft

\ifdraft 
    \usepackage[notcite,notref]{showkeys}

  \else
    \renewcommand{\sout}[1]{} %
\fi

\newcommand{\By}[2]{\overset{\mbox{\tiny{#1}}}{#2}}
\newcommand{\ByRef}[2]{   \By{\eqref{#1}}{#2} }

\newcommand{\geBy}[1]{    \By{#1}{\ge} }
\newcommand{\eqByRef}[1]{ \ByRef{#1}{=} }

\newcommand{\leByRef}[1]{ \ByRef{#1}{\le} }
\newcommand{\geByRef}[1]{ \ByRef{#1}{\ge} }
\newcommand{\justify}[1]{\fbox{\tiny{#1}}\quad}

\title{Prominent examples of flip processes}
\author{Pedro Ara\'ujo}
\author{Jan Hladk\'y}
\author{Eng Keat Hng}
\author{Matas \v{S}ileikis}
\email{araujo|hladky|hng|matas@cs.cas.cz}
\address{Institute of Computer Science of the Czech Academy of Sciences, Pod Vod\'arenskou v\v{e}\v{z}\'i~2, 182~07, Praha~8, Czechia. }

\thanks{With institutional support RVO:67985807. JH and EKH were supported by the Czech Science Foundation, grant number 21-21762X; PA and M\v{S} were supported by the Czech Science Foundation, grant number 20-27757Y}

\newcommand{\Dirac}{\boldsymbol{\delta}}
\newcommand{\sm}{\setminus}
\newcommand{\co}[1]{\overline{#1}}
\newcommand{\lgr}[1]{\mathcal{H}_{#1}}
\newcommand{\tlgr}[1]{\mathcal{F}_{#1}^{\bullet\bullet}}
\newcommand{\rul}{\mathcal{R}}
\newcommand{\Z}{\mathbb{Z}}
\newcommand{\R}{\mathbb{R}}
\newcommand{\N}{\mathbb{N}}
\newcommand{\ca}{\mathcal{A}}
\newcommand{\cb}{\mathcal{B}}
\newcommand{\cc}{\mathcal{C}}

\newcommand{\cf}{\mathcal{F}}
\newcommand{\cp}{\mathcal{P}}

\newcommand{\cu}{\mathcal{U}}
\newcommand{\eps}{\varepsilon}
\newcommand{\vphi}{\varphi}
\newcommand{\Prob}{\mathbb{P}}
\newcommand{\prob}[1]{\Prob \left( #1 \right)}
\newcommand{\Pc}[2]{\Prob \left(#1\,|\,#2\right)}
\newcommand{\E}{\mathbb{E}}
\newcommand{\Ec}[2]{\E \left(#1\,|\,#2  \right)}
\newcommand{\tind}{{t_{\mathrm{ind}}}}
\newcommand{\tr}[1]{t^{#1}}
\newcommand{\tindr}[1]{t_{\mathrm{ind}}^{#1}}
\newcommand{\Troot}[2]{T_{#1}^{#2}}
\newcommand{\cutn}[1]{ \left\lVert #1\right\rVert_{\square}}
\newcommand{\cutm}{\delta_\square}
\newcommand{\cutnd}{d_\square}
\newcommand{\Lone}[1]{ \left\lVert #1\right\rVert_{1}}
\newcommand{\Linf}[1]{ \left\lVert #1\right\rVert_{\infty}}
\newcommand{\tocutn}{\xrightarrow{\square}}
\newcommand{\toLinf}{\xrightarrow{\infty}}
\newcommand{\vel}[1][\:]{\mathfrak{V}_{#1}}
\newcommand{\traj}[2]{\Phi^{#1}#2}
\newcommand{\trajrul}[3]{\Phi_{#1}^{#2}#3}
\newcommand{\fl}[2]{\operatorname{Fl}^{\rul}_{#1}(#2)}
\newcommand{\vv}{\mathbf{v}}
\newcommand{\x}{\mathbf{x}}
\newcommand{\HH}{\mathbf{H}}
\newcommand{\D}{\,\mathrm{d}}
\newcommand{\indic}{\mathbb{I}}
\newcommand{\ind}[1]{\indic_{ \left\{ #1 \right\}}}
\newcommand{\Kernel}{\mathcal{W}}
\newcommand{\Gra}{\Kernel_{0}}
\newcommand{\sign}{\operatorname{sign}}
\newcommand{\G}{\mathbb{G}}
\newcommand{\hh}{\mathbb{H}}
\newcommand{\ff}{\mathbb{F}}
\newcommand{\bik}{\textstyle\binom{k}{2}}
\newcommand{\flo}[1]{ \left\lfloor #1 \right\rfloor}
\newcommand{\cei}[1]{ \left\lceil #1 \right\rceil}
\newcommand{\dest}{\mathrm{dest}}
\newcommand{\orig}{\mathrm{orig}}
\newcommand{\sdiff}{\ominus}
\newcommand\Bin{\operatorname{Bin}}
\newcommand\Po{\operatorname{Po}}
\newcommand{\age}[1]{\mathrm{age}(#1)}
\newcommand{\mdom}[1]{\mathcal{D}_{#1}}
\newcommand{\life}[1]{I_{#1}}
\newcommand{\downto}{\downarrow}

\def\endofClaim{\hfill\scalebox{.6}{$\blacksquare$}}
\newcommand{\oldqed}{}
\newenvironment{claimproof}[1][Proof]{
  \renewcommand{\oldqed}{\qedsymbol}
  \renewcommand{\qedsymbol}{\endofClaim}
  \begin{proof}[#1]
}{
  \end{proof}
  \renewcommand{\qedsymbol}{\oldqed}
}

\begin{document}
\begin{abstract}
	Flip processes, introduced in [\emph{Garbe, Hladk\'y, \v Sileikis, Skerman: From flip processes to dynamical systems on graphons}], are a class of random graph processes defined using a \emph{rule} which is just a function $\mathcal{R}:\mathcal{H}_k\rightarrow \mathcal{H}_k$ from all labelled graphs of a fixed order $k$ into itself. The process starts with an arbitrary given $n$-vertex graph $G_0$. In each step, the graph $G_i$ is obtained by sampling $k$ random vertices $v_1,\ldots,v_k$ of $G_{i-1}$ and replacing the induced graph $G_{i-1}[v_1,\ldots,v_k]$ by $\mathcal{R}(G_{i-1}[v_1,\ldots,v_k])$.
	
   Using the formalism of dynamical systems on graphons associated to each such flip process from \emph{ibid.}\ we study several specific flip processes, including the triangle removal flip process and its generalizations, `extremist flip processes' (in which $\mathcal{R}(H)$ is either a clique or an independent set, depending on whether $e(H)$ has less or more than half of all potential edges), and `ignorant flip processes' in which the output $\mathcal{R}(H)$ does not depend on $H$.
\end{abstract}
\maketitle

{\small
\tableofcontents}

\ifdraft %
\section*{\LaTeX{} macros}

\begin{center}
	\begin{tabular}{| l | c | p{8cm} |}
		\hline
		\verb"\Dirac_x" & $\Dirac_x$ & probability measure on a singleton $ \left\{ x \right\}$\\ \hline
		\verb"\sm" & $\sm$ & set difference, $A \sm B = A \cap \co{B}$ \\ \hline
		\verb"\co{A}" & $\co{A}$ & complement of a set; for a graph on $n$ vertices by $\co{G}$ we denote the graph $([n], E(K_n) \sm E(G))$ \\ \hline
		\verb"\lgr{k}" & $\lgr{k}$	& the set of labelled graphs on vertex set $[k] = \left\{ 1, \dots, k \right\}$ \\  \hline
		\verb"\tlgr{k}" & $\tlgr{k}$	& the set of isomorphism classes of two labelled vertices and $k$ vertices in total \\  \hline
		\verb"\rul" & $\rul$ & a rule, a square matrix indexed by $\lgr{k}$, which is stochastic, ie, for every $G \in \lgr{k}$ we have $\sum_{H \in \lgr{k}} \rul_{G,H} = 1$\\ \hline
		\verb"\N" & $\N$ & natural numbers, $1, 2, \dots$ \\ \hline
		\verb"\Z" & $\Z$ & integers \\ \hline
		\verb"\R" & $\R$ & real numbers \\ \hline
		\verb"\cf" & $\cf$ &  \\ \hline
		\verb"\eps" & $\eps$ & the nice epsilon \\ \hline
		\verb"\vphi" & $\vphi$ & the nice phi \\ \hline
		\verb"\Prob" & $\Prob$ & probability letter \\ \hline
		\verb"\prob{X \ge x}" & $\prob{X \ge x}$ & probability \\ \hline
		\verb"\Pc{A}{B}" & $\Pc{A}{B}$ & conditional probability \\ \hline
		\verb"\E" & $\E$ & expectation, mean \\ \hline
		\verb"\Ec{X}{B}" & $\Ec{X}{B}$ & conditional expectation \\ \hline
		\verb"\var" & $\var$ & variance \\ \hline
		\verb"\tind(G,W)" & $\tind(G,W)$ & induced density of graph $G$ in graphon $W$\\ \hline
		\verb"\tr{x}(G,W)" & $\tr{x}(G,W)$ & induced density of graph $G$ in graphon $W$ with partial vertices $J$ mapped to $x$ (we asume that $x$ is $J$-indexed so that there is no ambiguity\\ \hline
			\verb"\tindr{x}(G,W)" & $\tindr{x}(G, W)$ & induced density of a rooted graph $G = (R, V, E)$ in graphon $W$ with vertices $R$ mapped to $x$\\ \hline
			\verb"\cutn{W}" & $\cutn{W}$ & the cut norm \\ \hline
			\verb"\cutnd(U,W)" & $\cutnd(U,W)$ & the cut norm distance \\ \hline			
			\verb"\cutm(U,W)" & $\cutm(U,W)$ & the cut distance\\ \hline						
			\verb"\Troot{G}{a}" & $\Troot{G}{a}$ & operator of induced density of graph $G$  partial label vector $a$,  \\ \hline
		\verb"\Lone{f}" & $\Lone{f}$ & $L^1$-norm, $\int |f(x)| $ \\ \hline
		\verb"\Linf{f}" & $\Linf{f}$ & $L^\infty$-norm, $\operatorname{ess\,sup} f$ \\ \hline
		\verb"\tocutn" & $\tocutn$ & convergence in the cut norm $\cutn{\cdot}$ \\ \hline
		\verb"\toLinf" & $\toLinf$ & convergence in $L^\infty$-norm \\ \hline
	\end{tabular}
	\newpage
	\begin{tabular}{| l | c | p{8cm} |}
	  \hline
		\verb"\vel = \vel[\rul]" & $\vel = \vel[\rul]$ & a smooth function (field/operator) from $\Kernel$ to $\Kernel$ generated by rule $\rul$ so that $\vel W = \vel[\rul]W$ is the time derivative of $\traj{t}{W}$\\ \hline
		\verb"\traj{t}{W}" & $\traj{t}{W}$ & Trajectory/orbit of: position of $W$ at time $t$ \\ \hline
		\verb"\fl{t}{G}" & $\fl{t}{G}$ & flip process starting with graph $G$ after $t$ steps, Matas: so far I did not use it, find quite cumbersome \\ \hline
		\verb"\vv" & $\vv$ & a $k$-tuple chosen at a certain step of flip process \\ \hline
		\verb"\D" & $\D$ & differential symbol for integration \\ \hline
		\verb"\HH" & $\HH$ & the graph to which we flip at a certain step \\ \hline
		\verb"\x" & $\x$ & vector $(x_i : i \in S)$ for some set $S$ \\ \hline
		\verb"\indic_A" & $\indic_A$ & indicator (without subscript) \\ \hline
		\verb"\ind{x > 0}" & $\ind{x > 0}$ & indicator of a condition, which is put in curly brackets \\ \hline
		\verb"\Kernel" & $\Kernel$ & set of kernels  \\ \hline
		\verb"\Gra" & $\Gra$ & set of graphons, that is, $ \left\{ W \in \Kernel : W \in [0,1] \right\}$\\ \hline
		\verb"\Gra[\eps]" & $\Gra[\eps]$ & $\eps$-almost graphons\\
		\hline
		\verb"\sign" & $\sign$ &  \\ \hline
		\verb"\G(n,W)" & $\G(n,W)$ & a $W$-random graph: for every $i \in [k]$ sample independent $\pi$-distributed element $x_i$ from $\Omega$ and connect $i$ to $j$ with probability $W(x_i, x_j)$ \\ \hline
		\verb"\bik" & $\bik$ & $k$ choose $2$, the small typesetting (``text style'') \\ \hline
		\verb"\flo{x}" & $\flo{x}$ &  \\ \hline
		\verb"\cei{x}" & $\cei{x}$ &  \\ \hline
		\verb"\limcut" & $\limcut$ &  \\ \hline
		\verb"\liminfty" & $\liminfty$ &  \\ \hline
		\verb"A \sdiff B" & $A \sdiff B$ & symmetric difference of sets \\ \hline
		\verb"\Bin" & $\Bin(n,p)$ & binomial distribution \\ \hline
		\verb"\Po" & $\Po(\lambda)$ & Poisson distribution \\ \hline
		\verb"\mdom{W}" & $\mdom{W}$ & the maximal open interval on which $\traj{\cdot}{W}$ is defined \\ \hline
		\verb"\life{W}" & $\life{W}$ & times when the trajectory is inside the graphon space, $\{t \in \mdom{W} : \traj{t}{W} \in \Gra \}$ \\ \hline
		\verb"\age{W}" & $\age{W}$ & the age of graphon $W$, $\age{W} = -\inf\life{W}$ \\ \hline
		\hline
	\end{tabular}
\end{center}

\fi %

\section{Introduction}
\subsection{Random graph processes}\label{ssed:RGP}
Graphs are among the most universal structures, and they have therefore served as mathematical models in many settings. It is often the case that these settings evolve in time according to a stochastic rule. A variety of random graph processes have been introduced and studied extensively to this end. Most of these processes are defined so that they start with a specific initial graph on vertex set $[n]$ (for each $n\in\N$), and this vertex set remains unchanged. Further, it is usually asymptotic properties as $n$ tends to infinity which are investigated. Let us give some examples. The simplest such process is the \emph{Erd\H{o}s--R\'enyi graph process}, introduced in the seminal paper~\cite{MR125031} of Erd\H{o}s and R\'enyi in~1960. The process starts with the edgeless graph $G_0$ on $[n]$. In each step $t=1,\ldots,\binom{n}{2}$, the graph $G_t$ is obtained by taking $G_{t-1}$ and turning a uniformly random non-edge into an edge. Obviously, the graph $G_{\binom{n}{2}}$ is complete. Another process introduced by Bollob\'as and Erd\H{o}s at a ``conference in Fairbanks, Alaska, in 1990, while admiring the playful bears'' \cite{MR1615568}, is the \emph{triangle removal process}. The process starts with the complete graph $G_0$ on $[n]$. In each step $t=1, 2, \ldots$, the graph $G_t$ is obtained by taking $G_{t-1}$, choosing a randomly selected triangle (uniformly among the set of all triangles in $G_{t-1}$), and erasing its three edges. The process stops at the (random) time when $G_t$ is triangle free. A major question raised already by Bollob\'as and Erd\H{o}s in 1990 was what the number of edges of the final graph $G_t$ is. This has been a subject of a long line of research, culminating in~\cite{MR3350225} which establishes that $e(G_t)=n^{3/2+o(1)}$. General $H$-removal processes were studied much less, mostly motivated by extending the setting of the Erd\H{o}s--Hanani conjecture resolved in~\cite{Rodl:ErdHan}, see e.g.~\cite{PippSpe:Asymptotic,Spe:AsymptoticPacking}. A process which looks like a dual process in many aspects is the following one, called the \emph{$H$-free process}. Let $H$ be a fixed graph. We start with the edgeless graph $G_0$. In each step~$t$ we select, uniformly at random, a non-edge whose addition to $G_{t-1}$ does not create a copy of $H$, and turn it into an edge.  Hence, when the process stops, $G_t$ is a maximal $H$-free graph. The case $H=K_3$ is now well understood due to analyses by Bohman and Keevash~\cite{MR4201797} and Fiz Pontiveros, Griffiths, and Morris~\cite{MR4073152}. In particular, these results imply the currently best known lower bound for the off-diagonal Ramsey number $R(3,k)$. Among the numerous further papers studying the triangle-free process and other $H$-free processes, let us highlight~\cite{MR1799803,MR2657427,MR3214202}. The last class of processes we mention before turning to the main subject of this paper are the \emph{Achlioptas processes}, introduced in~\cite{MR2502843}. We start with the edgeless graph on $[n]$. In each step we select $k=2$ random non-edges (natural generalizations for larger $k$ exist) $f_1=x_1y_1$ and $f_2=x_2y_2$. Then, we choose one of them to add into the graph according to a predetermined rule. For example, the so-called `product rule' selects the non-edge $f_i$ which minimizes the product of the sizes of the connected components to which $x_i$ and $y_i$ belong. Indeed, the motivation behind many Achlioptas rules was to alter the time of the emergence of the giant component. See, e.g.\ \cite{MR2985166,MR3366817}, for more on Achlioptas processes.

Let us note that an obvious generalization of the processes above would be to start them from graphs other than the complete or the edgeless one. It seems that the reason such generalizations were not studied was that it was not clear what the observed features should be.

\subsection{Flip processes, basics}
In~\cite{Flip1}, the following general class of random graph processes was introduced. Let $k\in\N$ be given and let $\lgr{k}$ be the set of all graphs on the vertex set $[k]$. Let $\rul \in [0,1]^{\lgr{k}\times \lgr{k}}$ be a matrix which is stochastic, that is, for every $F \in \lgr{k}$ we have
$\sum_{H\in \lgr{k}}\rul_{F,H}=1$. We call $k$ the \emph{order} and $\rul$ the \emph{rule} of the flip process. Given $\rul$, an integer $n \ge k$ and an \emph{initial graph} $G_0$ with vertex set $[n]$, the corresponding \emph{flip process} proceeds as follows. For each $t\in \N$, we sample uniformly at random an ordered tuple $\mathbf{v} = (v_1,\cdots,v_k)$ of distinct vertices in $[n]$. Let $F\in\lgr{k}$ be the \emph{drawn graph}, that is, the graph on $[k]$ such that $ij \in E(F)$ if and only if $v_iv_j\in G_{t-1}$. Next, we sample a graph $H$, which we term the \emph{replacement graph}, according to the probability distribution $(\rul_{F,H})_H$ and replace $G_{t-1}[\mathbf{v}]$ by $H$ in the graph $G_{t-1}$ to obtain $G_t$. To be precise, we place a copy of $H$ on $\mathbf{v}$ so that $ij \in E(H)$ if and only if $v_iv_j \in G_t$. A flip process is a discrete time Markov process. It is (unlike all the previously mentioned processes) defined for all times.

We observe that the Erd\H{o}s--R\'enyi process and the triangle removal process have flip process variants. The \emph{Erd\H{o}s--R\'enyi flip process} starts with an arbitrary graph on $n \ge 2$ vertices. In each step, a random pair of vertices is chosen and turned into an edge. If the pair is already an edge then the step is idle. The \emph{triangle removal flip process} starts with an arbitrary graph on $n \ge 3$ vertices. In each step, a random triple of vertices is chosen. If the triple forms a triangle then its three edges are removed. Otherwise, the step is idle. Note that the difference between the original processes and their flip process variants is the idle steps in the latter. While $H$-removal processes can be viewed as flip processes, other classes mentioned in Section~\ref{ssed:RGP}, like the $H$-free process and Achlioptas processes, cannot.

In~\cite{Flip1}, flip processes were studied from the perspective of dense graph limits. Recall that the main objects of this theory, introduced in~\cite{Lovasz2006,MR2455626}, are two-variable symmetric measurable functions defined on the square of a probability space $(\Omega,\pi)$ (which we assume to have a separable sigma-algebra and atomless measure) with values in $[0,1]$. These functions are called \emph{graphons}. We denote the set of all graphons by $\Gra$. Recall that one of the most important metrics on $\Gra$ is the \emph{cut norm distance}, which is given by
\begin{equation}\label{eq:defcnd}
\cutnd(U,W) =\cutn{U-W}=\sup_{S,T\subseteq\Omega} \left|\int_{S\times T}(U-W)\D\pi^2\right|\;.
\end{equation}
Recall also that each finite graph $G$ can be represented as a graphon. To this end, take an arbitrary partition $(\Omega_v:v\in V(G))$ of $\Omega$ into sets of equal measure. Let the graphon $W_G$ be defined to be~$1$ on each rectangle $\Omega_u\times \Omega_v$ for which $uv\in E(G)$, and~0 elsewhere. Note that this graphon representation depends on the choice of the partition $(\Omega_v:v\in V(G))$. However, we will assume that graphs on the same vertex set are represented using the same partition of $\Omega$.
Since each step of a flip process of order $k$ flips at most $\binom{k}{2}$ pairs of vertices (from a non-edge to an edge or vice versa), it follows that we need to perform $\Omega(n^2)$ steps to see a change in the evolution of an $n$-vertex graph through the lens of dense graph limits. Indeed, one of the main results of~\cite{Flip1} is the following correspondence.
\begin{thm}[Informal, with loose quantification; see Theorem~\ref{FIRSTPAPER.thm:conc_proc} in~\cite{Flip1}]\label{thm:correspondenceInformal}
  For every $k\in\N$ and for every rule $\rul$ of order $k$, there exists an \emph{evolution function} $\traj{}{}:\Gra\times [0,\infty)\rightarrow\Gra$ so that, writing $\traj{t}{W} = \Phi(W, t)$ and treating $\Gra$ as the metric space with metric $\cutnd$, we have the following.
\begin{romenumerate}
  \item For any fixed graphon $W\in \Gra$, we have $\traj{0}W = W$ ($W$ is the `initial graphon') and the map $t \mapsto \traj{t}{W}:[0,\infty)\rightarrow\Gra$, which we call the \emph{trajectory starting at $W$}, is continuous.
  \item The trajectory depends continuously on the initial conditions, i.e., for any fixed time $t\ge 0$, the map $\traj{t}{(\cdot)}:\Gra\rightarrow\Gra$ is continuous. The variable $t$ is `time'.
  \item \label{item:thm-corr-traj-flip-transfer} Let $G_0, G_1, \dots$ be the flip process starting with an $n$-vertex graph $G_0$. Denote by $W_\ell$, $\ell = 0, 1, \dots$, the (random) graphon that represents the graph $G_\ell$. For every $\eps>0$ and $T\ge 0$, with probability $1-\exp(-\Omega(n^2))$ for every $\ell\in\N\cap [0,Tn^2]$ we have $\cutnd(W_\ell,\traj{\ell/n^2}{W_0})<\eps$.
\end{romenumerate}
\end{thm}

The evolution function $\traj{}{}:\Gra\times [0,\infty)\rightarrow\Gra$ is constructed in a way which is only partly explicit. More specifically, given a rule $\rul$, a \emph{velocity field} $\vel:\Gra\rightarrow \Kernel$ is explicitly constructed (here, $\Kernel$ is the space of 2-variable symmetric functions in $L^\infty(\Omega^2)$),
and the trajectory starting at any initial graphon $W\in\Gra$ is then defined as the unique solution of the (Banach-space-valued) differential equation
\[
  \frac{\D}{\D t}\traj{t}{W} = \vel \left(\traj{t}{W}\right)\; \mbox{with initial condition} \; \traj{0}{W}=W\;,
\]
which can be equivalently formulated as the integral equation
\begin{equation}\label{eq:integralequation}
  \traj{t} W = W + \int_0^t \vel \left(\traj{\tau} W\right) \D \tau\;.
\end{equation}
As we shall see in this paper, obtaining explicit trajectories is a difficult task even for fairly simple rules.

In the rest of this introduction we motivate further key notions in the theory of flip processes.

\subsection{Flip processes, going back in time}
Theorem~\ref{thm:correspondenceInformal}\ref{item:thm-corr-traj-flip-transfer} says that for any time $t \ge 0$ the graphon $\traj{t}{W}$ approximately describes the structure of a typical graph after $tn^2$ steps of the flip process started at an $n$-vertex graph $G_0$ which is close to $W$. We may ask a complementary question: does there exist a graphon $U$ such that, starting at an $n$-vertex graph $H_0$ that is close to $U$, after $tn^2$ steps we typically arrive to a graph which is similar to $W$? In such a case, it is natural to write $\traj{-t}{W} := U$, and the above property can be written as $\traj{t}{(\traj{-t}{W})}=W$. Note that $\traj{-t}{W}$ need not always exist. To give a specific example, in the Erd\H{o}s--R\'enyi flip process, there is no way we could arrive at an edgeless graph after performing any number of steps starting from any graph. This means that $\traj{-t}{0}$ does not exist for any $t>0$. As a generalization of this example, for any constant graphon $C\in(0,1)$, the backwards trajectory $\traj{-t}{C}$ is a constant graphon which decreases as $t$ increases until it reaches~0 (at time $t=-\frac12\ln(1-C)$; this quantity can be obtained by considering the inverse of the trajectory starting at zero, see~\cite[equation \eqref{FIRSTPAPER.eq:solutionErdosRenyi}]{Flip1}), and for larger $t$, $\traj{-t}{C}$ is not defined. Switching back to the general setting, we call the supremum of times $t$ for which $\traj{-t}{W}$ exists the \emph{age of a graphon $W$}. Furthermore, if $\age{W}<\infty$ then it can be shown that $\orig(W) := \traj{-\age{W}}{W}$ is well-defined (see Theorem~\ref{FIRSTPAPER.thm:flow}\ref{FIRSTPAPER.en:life} in~\cite{Flip1}). We call $\orig(W)$ the \emph{origin} of $W$. To return to the example of the Erd\H{o}s--R\'enyi flip process, the origin of every constant graphon $C\in[0,1)$ is the constant-$0$. The constant-$1$ does not have an origin, but it has an infinite age and $\traj{t}{1}=1$ for every time $t$, positive or negative.

\subsection{Flip processes, destinations}
For a general rule and a fixed initial graphon $W$, there are several qualitatively different possible evolutions of the trajectory $(\traj{t}{W})_{t\ge 0}$. The simplest one is when $W$ is a \emph{fixed point}, that is, $\traj{t}{W}=W$ for all $t\in\R$. The next simplest is when $\traj{t}{W}$ converges as $t\to\infty$. Here we choose the weakest sensible notion of convergence, that is, cut norm distance convergence.\footnote{However, in~\cite{Flip1} it is conjectured that this setting actually implies convergence in the $L^1$-norm and possibly even in the $L^\infty$-norm.} In such a setting we denote the limit by $\dest(W)$ and call it the \emph{destination of} $W$. There are more complicated types of trajectories. For example, in~\cite{Flip1} it was shown that a periodic trajectory exists for a certain rule. Further, when a starting graphon is taken in an proximity of such a periodic trajectory, its trajectory spirals towards the periodic one. However, at this moment we do not see a classification of trajectories beyond convergent/nonconvergent (where the former also contains fixed points).

For a convergent trajectory, it is of interest to investigate the speed of convergence. Bounds which are independent of the starting graphon are particularly interesting. More formally, suppose that $\rul$ is a rule for which every trajectory converges. We can define for each $\delta>0$ the \emph{lower convergence time}
\[
  \tau^-(\delta)=\sup \{\inf\{T:\cutnd(\traj{T}{W},\dest(W))<\delta \}:W\in \Gra\}
\]
and the \emph{upper convergence time}
\[
  \tau^+(\delta) = \sup \{\sup\{T:\cutnd(\traj{T}{W},\dest(W))\ge \delta \}:W\in \Gra\} \;.
\]

\subsection{Our results}
This paper investigates properties of some natural flip processes, with a particular emphasis on the description of their trajectories. Due to the key Theorem~\ref{thm:correspondenceInformal}, this provides information about the typical evolution of an $n$-vertex graph in the next $\Theta(n^2)$ steps.

Let us briefly and informally introduce these flip processes and highlight the main results we obtain about them. The main results that concern either explicit description of trajectories, destinations, speed of convergence or speed of these flip processes are listed in Table~\ref{t:summary}. (In addition to those, some flip processes have specific features.)
\begin{table}[t]
	\begin{tabular}{ | l | c | c | c | c | }
		\hline 		
		& trajectories  & destination  & age  & speed\tabularnewline
		\hline
		ignorant  & \multicolumn{4}{c|}{Proposition~\ref{prop:ignorant-behaviour}} \tabularnewline
		\hline 
		balanced stirring  & \multicolumn{2}{c|}{Proposition~\ref{prop:stirring-trajectory}}&Proposition~\ref{prop:agestirring} & \tabularnewline
		\hline 
		complementing  & \multicolumn{4}{c|}{Proposition~\ref{prop:complementing-behaviour}} \tabularnewline
		\hline 
		\multirow{2}{*}{extremist}  &  & Theorem~\ref{thm:extremist-threshold-bound} &  & \tabularnewline
		&  & Theorem~\ref{thm:extremist-zero-threshold-bound} &  & \tabularnewline
		\hline 
		\multirow{2}{*}{monotone}  &  & Proposition~\ref{prop:monotoneflipprocess}  &  & \tabularnewline
		 &  & Proposition~\ref{prop:NondecreasingSinksLower}  &  & \tabularnewline
		\hline 
		component completion  &  & Theorem~\ref{thm:component-completion-dests}  &  & \tabularnewline
		\hline 
		\multirow{2}{*}{removal } &  & Fact~\ref{fact:removaldests} &  & Theorem~\ref{thm:speedremoval}\tabularnewline
		&  & Theorem~\ref{thm:removalprocSinkLimit} &  & Proposition~\ref{prop:removalBipVSNonbip} \tabularnewline
		\hline 
	\end{tabular}
	\caption{Summary of our main (and easily categorizable) results}
	\label{t:summary}
\end{table}

\subsubsection{Ignorant flip processes (Section~\ref{sec:ignorant})}
Ignorant flip processes are flip processes where the distribution for the replacement graph does not depend on the drawn graph. That is, the rule $\rul$ satisfies that for any $F,F',H$, we have $\rul_{F,H} =\rul_{F',H}$. Ignorant flip processes are one of the few examples where we can describe the trajectories explicitly; this gives a complete description of their behaviour. In particular, all the trajectories of an ignorant flip process converge to the same constant graphon whose value is the expected edge density\footnote{The \emph{density} of a graph $H$ is defined as ${e(H)}/{\binom{v(H)}{2}}$.} according to the output distribution $\{\rul_{F,H}\}_H$ (here $F$ is arbitrary). The Erd\H{o}s--Renyi flip process is a prominent example of an ignorant flip process.

\subsubsection{Stirring flip processes (Section~\ref{sec:stirring})}
While the output distribution in an ignorant flip process does not depend on the drawn graph at all, in a stirring flip process it depends only on a very basic graph parameter --- its number of edges. In Section~\ref{sec:stirring} we investigate balanced stirring flip processes, in which the expected number of edges in the replacement graph is equal to the number of edges in the drawn graph. We explicitly describe the trajectories of any balanced stirring flip process. Not surprisingly, the edge density is preserved along each trajectory. The description does not depend on the actual rule but only on its order. That is, any balanced stirring flip process of a given order yields the same trajectories.

\subsubsection{Complementing flip processes (Section~\ref{sec:complementing})}
The complementing flip process is a flip process where we replace the drawn graph by its graph complement. Informally, we thus sparsify the host graph at places where its density lies above $\frac12$ and densify it at places where its density lies below $\frac12$. It follows that each trajectory should converge to the constant-$\frac12$ graphon. We work out the trajectories explicitly, confirming this intuition.

\subsubsection{Extremist flip processes (Section~\ref{sec:extremist})}
The extremist flip process is defined as follows. If the drawn graph has edge density less than $1/2$, we replace it by the edgeless graph. If the edge density is greater than $1/2$, we replace it by the clique. If the density is exactly $1/2$, then there is some freedom in defining the replacement graph (our default choice is do nothing). Intuitively, dense graph(on)s converge to the complete graph (constant-$1$ graphon) while sparse graph(on)s converge to the edgeless graph (constant-$0$ graphon). The main focus of our treatment of extremist flip processes is in making this precise. Further, we study the limit behaviour as $k$ grows.

\subsubsection{Monotone flip processes (Section~\ref{sec:monotone})}
We introduce several variants, but, roughly speaking, a monotone flip process is a flip process where in each step we add edges, or at least add edges in expectation. In such a situation we expect the trajectories to be increasing in time and in particular to converge. Indeed, most of our treatment of monotone flip processes focuses on these features (including one negative example).

\subsubsection{Component completion flip processes (Section~\ref{sec:componentcompletion})}
In a component completion flip process, we replace the drawn graph (which may have several connected components) by its so-called component completion graph, that is, we turn each connected component into a clique. Hence, the component structure of the graph does not change during the evolution. Note also that this procedure is void for orders $k=1,2$. For $k\ge3$ we prove that the trajectories converge to a graphon with the same component structure, where each component is constant-$1$.

\subsubsection{Removal flip processes (Section~\ref{sec:RemovalProces})}
Removal flip processes generalize the triangle removal flip process. We fix a graph $F\in\lgr{k}$. In each step of the process, we delete a copy of $F$, should the drawn graph contain it. More formally, for $H\in\lgr{k}$ which contains $F$ as a subgraph (preserving the labels) we have $\rul_{H,H-F}=1$ while for $H$ which does not contain $F$ as a labelled subgraph we have $\rul_{H,H}=1$. The theory of removal flip processes is particularly elegant. There are two results which we could obtain only with the help of Szemer\'edi's regularity lemma: that convergence times for removal processes are finite and uniformly bounded and that the map which maps each graphon to its destination is continuous in the cut norm distance.

\subsubsection{Quadratically many steps and beyond}
Most of the results we obtain concern trajectories. Due to the correspondence in Theorem~\ref{thm:correspondenceInformal}, this tells us about the typical evolution of a finite $n$-vertex graph in a flip process for $\Theta(n^2)$ steps. However, for the extremist and removal flip processes we found interesting questions which go beyond this time horizon.

\section{Definitions and basics}
This section recalls some basic facts, which we sometimes reprove in variants that are tailored to our applications.

For $k\in \N$, we write $(k)_2=k(k-1)$. When using little-o notation, we sometimes use a subscript to emphasize the parameter of asymptotics. For example, $\lim_{k\to\infty}\frac{o_k(k^2)}{k^2}=0$.

\subsection{Differential equations}

\subsubsection{Gr\"{o}nwall's inequality}

Here we state two variants of Gr\"{o}nwall's inequality. The following lemma is a variant of Gr\"{o}nwall's inequality in integral form.

\begin{lem}[{\cite[Theorem 1.10]{tao2006nonlinear}}]\label{lem:Groenwall}
	Let $f : [t_0, t_1] \rightarrow [0, \infty)$ be a continuous non-negative function.
	Suppose that for some constants $A\ge 0$, $C > 0$ we have
	\begin{equation}
	\label{eq:Viggen}
	f(t) \le A + C \int_{t_0}^t f(x) \D x
	\end{equation}
	for every $t \in [t_0, t_1]$. Then we have
	\begin{equation}
	f(t) \le A e^{ C (t-t_0) } \quad \text{ for all } t \in [t_0, t_1]\;.
	\end{equation}
\end{lem}

The following lemma is a variant of Gr\"{o}nwall's inequality for differentiable functions where the differential inequality holds subject to a condition on the value of the function.

\begin{lem} \label{lem:FODEonesidedBound}
Let $D,K>0$ and $f \colon [0,\infty)\to[0,\infty)$ be a differentiable function satisfying $f(0) \le D$ and $f'(t) \le -Kf(t)$ for all $t \in [0,\infty)$ such that $f(t) \le D$. Then for all $t\in[0,\infty)$ we have $f(t) \le f(0)e^{-Kt}$.
\end{lem}

\begin{proof}
First we show that we have $f(t) \le D$ for all $t \in [0,\infty)$. Indeed, suppose that we have $f(t')>D$ for some $t'>0$. Let $T := \inf\{t \ge 0 : f(t) > D\}$. By the continuity of $f$ we have $f(T) = D$, so we have $f'(T) \le -Kf(T) = -KD < 0$. By the definition of $f'(T)$, there exists $\eps_0 > 0$ such that for all $0<\eps<\eps_0$ we have
\[
  f(T+\eps) \le f(T) +\eps(KD/2 + f'(T)) \le f(T) - KD\eps/2 < D.
\]
This contradicts the definition of $T$, so we have $f(t) \le D$ for all $t \in [0,\infty)$. 

The function $g \colon [0,\infty) \to [0,\infty)$ given by $g(s) := f(s)e^{Ks}$ is differentiable with $g'(s) = f'(s)e^{Ks} + Kf(s)e^{Ks} \le 0$, so for every $t \ge 0$ we have $g(t) \le g(0) = f(0)$ and hence $f(t) = g(t)e^{-Kt} \le f(0)e^{-Kt}$, as desired.
\end{proof}

\subsubsection{On notation of integrals}
For the sake of brevity and clarity we sometimes use simplified notation to denote an integral, namely we often write $\int_{x = a}^b f(x)$ (or $ \int_{x = a}^b f$) instead of $\int_a^b f(x) \D x$ for functions of a single real argument (that is, the underlying measure is the Lebesgue measure) and $\int_{x \in A} f(x)$ (or $\int_{x \in A} f$) instead of $\int_A f(x) \D \mu(x)$ for functions on some abstract measure space when the measure $\mu$ is clear from the context.

\subsubsection{Differential equations in Banach spaces} \label{subsubsec:diff-eqn-Banach}
We need to recall concepts of the derivative and the integral in Banach spaces. However we do not need almost any results from the abstract theory of differential equations in Banach spaces since we only need to recall results obtained already in~\cite{Flip1}. Let $\mathcal{E}$ be a Banach space with norm $\|\cdot\|$. Suppose that a function $f : I \to \mathcal{E}$ is defined for an open set $I \subset \R$. If $t\in I$ is such that there is an element $z \in \mathcal{E}$ for which
\[
\lim_{\eps \to 0} \left\| \frac{f(t +\eps) - f(t)}{\eps} - z \right\| = 0\;,
\]
then we say that \emph{$f$ is differentiable at $t$} and call $\frac{\D }{\D t}f(t)= z$ the \emph{derivative of $f$ at $t$}. At the heart of the theory of flip processes are first-order autonomous differential equations of the type
\begin{equation}
\label{eq:autonomousdifferentialequation}
\frac{\D}{\D t}f(t) = \mathfrak{O}(f(t)) \; \mbox{with initial condition} \; f(0)=z\;,
\end{equation}
where $z\in \mathcal{E}$, $f$ takes values in $\mathcal{E}$, and $\mathfrak{O}:\mathcal{E}\rightarrow \mathcal{E}$ is (locally) Lipschitz. Such a differential equation can be rewritten as the integral equation
\begin{equation}
\label{eq:abstractintegralequation}
f(t) =z+\int_{0}^t \mathfrak{O}(f(s))\D s\;,
\end{equation}
where the integral is the so-called Cauchy--Bochner integral in the Banach space $\mathcal{E}$ (it is easy to show that this integral is well-defined, since $s \mapsto \mathfrak{O}(f(s))$ is continuous). The integral has all usual properties of the Riemann integral of continuous real-valued functions, with the norm playing the role of the absolute value (see Section~\ref{FIRSTPAPER.app:Banach_calc} in~\cite{Flip1}).

In this paper, we specifically consider differential equations in the Banach space $(\Kernel, \Linf{\cdot})$, which is a closed subspace of $L^\infty(\Omega^2)$ (where $\Omega$ is the probability space on which our graphons are defined). Recall that the elements of $L^\infty(\Omega^2)$ are equivalence classes (where equivalent means equal almost everywhere) of measurable functions, which means that it is not formally meaningful to evaluate an element $f \in L^\infty(\Omega^2)$ at a given point $(x,y) \in \Omega^2$. However, it is much more intuitive to pretend we can do this and, for example (i) evaluate $f(t)$ (see~\eqref{eq:integralequation} for a prominent example) at point $(x,y)$ by integrating a real-valued curve $s \mapsto \vel \left( f(s) \right)(x,y)$; (ii) deduce that $f(t) \le g(t)$ for every $t \in [a,b]$ implies $\int_a^b f(t) \le \int_a^b g(t)$ (where the inequalities are understood to hold almost everywhere). In~\cite{Flip1} it was shown that this apparently non-rigorous approach can be perfectly justified. We believe that this approach of dealing with trajectories ``pointwise'' has an advantage on the intuitive level, and we stick to it, referring a meticulous reader to the required technical tools for justification in Appendix~\ref{FIRSTPAPER.app:Banach_calc} in \cite{Flip1}.

What follows is the only abstract fact we need that is not available in \cite{Flip1}. It says that a solution of a differential equation of the type in~\eqref{eq:autonomousdifferentialequation} changes only little if we perturb the velocity field. This statement is very likely standard, but we found it easier to give a four-line proof than to find a source in the literature.
\begin{lem}
	\label{lem:perturb}
	Consider a Banach space $\mathcal{E}$ with norm $\|\cdot\|$. Let $\mathcal{U} \subset \mathcal{E}$ be a set, and $A,K > 0$ be constants. Let $X,Y : \mathcal{U} \rightarrow \mathcal{E}$ be functions. Suppose that $X$ is $K$-Lipschitz, that is, we have
\begin{equation}\label{eq:AMo}
	\left\|X(x)-X(y)\right\| \le K \left\|x-y\right\| \quad\text{ for all } x,y\in \mathcal{U}\;.
\end{equation}
	Suppose further that for each $x\in\mathcal{U}$ we have
\begin{equation}\label{eq:BMo}
\|X(x)-Y(x)\| \le A\;.
\end{equation}
	
Given $x_0\in\mathcal{U}$, let $\alpha,\beta:[0,T]\rightarrow \mathcal{U}$ satisfy the integral equations
	\begin{equation*}
	\alpha(t)=x_0+\int_{s=0}^t X(\alpha(s))\quad \beta(t)=x_0+\int_{s=0}^t Y(\beta(s))\quad\mbox{for each $t\in[0,T]$}
	\end{equation*}
	Then for each $t\in[0,T]$ we have $\|\alpha(t)-\beta(t)\| \le AT\exp(K t)$.
\end{lem}
\begin{proof}
	Set $f(t) := \|\alpha(t)-\beta(t)\|$. Then
\begin{align*}
f(t)& \le \int_{s=0}^t \left\|X(\alpha(s))-Y(\beta(s))\right\| \leByRef{eq:BMo}
\int_{s=0}^t \big(A+ \left\|X(\alpha(s))-X(\beta(s))\right\|\big)\\
& \leByRef{eq:AMo} \int_{s=0}^t \big(A + Kf(s)\big) \le AT+K\int_{s=0}^t f(t)\;.
\end{align*}
Hence, the claim follows from Lemma~\ref{lem:Groenwall}.
\end{proof}

\subsection{Graphons}
Our notation regarding graphons follows~\cite{Lovasz2012}. Recall that $(\Omega,\pi)$ is an atomless probability space with an implicit separable sigma-algebra. We write $\Kernel\subset L^\infty(\Omega^2)$ for the set of bounded symmetric measurable functions; these are called \emph{kernels}. \emph{Graphons}, which we denote by $\Gra$, are then those kernels whose range is a subset of $[0,1]$.

\subsubsection{Densities and degrees}
The notion of homomorphism densities of a finite graph is one of the cornerstones of the theory of graph limits.

For a function $f : \Omega^V \to \R$ and a subset $S \subset V$ we write
\begin{equation*}
	\int f \D\pi^S=\int_{\Omega^S} f(x_i : i \in V) \prod_{i \in S} \pi(\D x_i).
\end{equation*}
\begin{defi}
	Given a kernel $W$ and a graph $F = (V,E)$, we define the \emph{density of $F$} in $W$~as
	\begin{equation*}
	t(F,W) \coloneqq \int \prod_{ij \in E} W(x_i,x_j) \D\pi^{V} \;.	
	\end{equation*}
	and the \emph{induced density of $F$} in $W$ as
	\begin{equation*}
	\tind(F,W) \coloneqq \int \prod_{ij \in E} W(x_i,x_j) \prod_{ij \notin E} \left( 1 - W(x_i,x_j) \right) \D\pi^{V}\;.	
	\end{equation*}
    In particular, the \emph{edge density of $W$} is the number $t(K_2,W) = \int W \D\pi^2$.
\end{defi}
A \emph{rooted graph} is a triple $(R,V,E)$, where $(V,E)$ is a graph and $R$ is a subset of $V$; we call the elements of $R$ the \emph{roots}.
\begin{defi}[Rooted densities]
  \label{def:densities}
  Given a graphon $W$, a rooted graph $F = (R, V, E)$, and ${\x} = (x_i : i \in R) \in \Omega^R$ we define the functions
	\begin{align*}
	\tr{\x}(F,W) &\coloneqq \int \prod_{ij \in E} W(x_i,x_j) \D \pi^{V \sm R}\;\textrm{ and}\\
	\tindr{\x}(F,W) &\coloneqq \int \prod_{ij \in E} W(x_i,x_j) \prod_{ij \notin E} \left( 1 - W(x_i,x_j) \right) \D \pi^{V \sm R}\;.
	\end{align*}
   Further, if $\mathbf{A}=(A_i)_{i\in R}$ is a collection of sets $A_i\subset \Omega$ then we define the quantities
   	\begin{equation*}
	  \tr{\mathbf{A}}(F,W) \coloneqq \int_{\x \in \prod_i A_i} \tr{\x}(F,W) \quad \text{and} \quad \tindr{\mathbf{A}}(F,W) \coloneqq \int_{\x \in \prod_i A_i} \tindr{\x}(F,W)\;.
   \end{equation*}
\end{defi}

One of the key properties of densities is that they, as functions of a graphon, are continuous with respect to the cut norm distance.
\begin{lem}[Counting lemma, {\cite[Lemma~10.23]{Lovasz2012}}]\label{lem:countinglemma}
	Suppose that we have a graph $F$ of order $k$ and two graphons $U$ and $W$. Then we have
	\[
	  |t(F,U) - t(F,W)| \le \bik\cutnd(U, W).
	\]
\end{lem}

Lastly, we introduce the notions of degree and degree-regularity.
\begin{defi}
  The \emph{degree function} of a kernel $W$ is the function $\deg_W : \Omega \to \R$ defined as $\deg_W(x) = \int_{y \in \Omega} W(x,y)$. We say that a kernel $W$ of edge density $d$ is \emph{degree-regular} if for $\pi$-almost every $x\in \Omega$ we have $\deg_W(x)=d$.
\end{defi}

\subsubsection{Sampling}
\label{ss:sampling}
A graphon $W$ defines, for every $k\in \N$, a probability distribution on $\lgr{k}$ as follows. Sample independent $\pi$-random elements $U_v$, $v \in [k]$ of $\Omega$ and then connect each pair of distinct vertices $u, v \in [k]$ independently with probability $W(U_u, U_v)$. We write $F \sim \G(k,W)$ for the random graph obtained this way and sometimes even use $\G(k,W)$ to denote the random graph itself. For many properties of sampling we refer to Section~10 of \cite{Lovasz2012}. The following lemma asserts that for large $k$ this random graph is close to $W$ with high probability. To this end we use the notion of \emph{cut distance} between a finite graph $G$ and a graphon $W$, which is defined as $\cutm(G,W)=\inf\{\cutnd(U,W)\}$, where $U$ ranges over all graphon representations of $G$.
\begin{lem}[Sampling lemma, {\cite[Lemma~10.16]{Lovasz2012}}]
  \label{lem:sampling}
  For $k = 2, 3, \dots$ and every $W\in\Gra$ we have
\[
  \prob{\cutm(\G(k,W),W) \le 22/\sqrt{\log_2 k} } \ge 1-\exp(-k/(2\log_2 k)) \;.
\]
\end{lem}

\subsubsection{Regularity lemma}
In Section~\ref{sec:RemovalProces} we study removal flip processes. To study their speed of convergence, we need to recall a regularity lemma. While for many purposes in the area of graph limits the Frieze--Kannan regularity lemma is sufficient, here we need Szemer\'edi's regularity lemma. The notions of density and regularity are standard, but typically appear in the setting of finite graphs. The graphon setting is used for example in~\cite{MR2306658}.
\begin{defi}
  Suppose that $W$ is a graphon and $C,D\subset \Omega$ are sets of positive measure. The \emph{density of $(C,D)$} in $W$ is defined as 
  \[
    d(C,D) := \frac{1}{\pi(C)\pi(D)}\int_{C\times D}W \D \pi^2 \; .
  \]
  Given $\alpha>0$, we say that $(C,D)$ is \emph{$\alpha$-regular} in $W$ if for every $C'\subset C$ and $D'\subset D$ we have
\begin{equation}\label{eq:IAmRegular}
  \left|\pi(C')\pi(D') d(C,D) - \int_{C'\times D'}W \D \pi^2\right| \le \alpha\pi(C)\pi(D)\;.
\end{equation}
\end{defi}
\begin{defi}
	Suppose that $W$ is a graphon. For $\alpha>0$, we say that a partition $\{P_1,\ldots,P_M\}$ of $\Omega$ is an \emph{$\alpha$-regular partition of $W$} if
	\begin{itemize}
		\item $\alpha\ge 1/M$,
		\item for each $i$ we have $\pi(P_i)=1/M$,
		\item the number of pairs $(i,j)\in[M]^2$ for which $(P_i,P_j)$ is not an $\alpha$-regular pair in $W$ is at most $\alpha M^2$.
	\end{itemize}
	We call the sets $P_1,\ldots,P_M$ \emph{clusters}.
\end{defi}
With this terminology, we can state a version of Szemer\'edi's regularity lemma tailored to our purposes.
\begin{thm}\label{thm:RL}
  For each $\alpha>0$ there exists $M_0\in \N$ so that for each graphon $W$ there exists an $\alpha$-regular partition of $W$ with at most $M_0$ sets.
\end{thm}
In the setting of regular pairs, homomorphism densities can be controlled. This is the subject of the counting lemma, which we state in a version tailored to our purposes.
\begin{lem}\label{lem:RLcounting}
  Let $\alpha>0$ and $W$ be a graphon. Let $F = (R, R, E)$ be a graph which satisfies $|R| \ge 3$ and whose vertices are all roots. Suppose that $\mathbf{A}=(A_i)_{i\in R}$ be a tuple of subsets of $\Omega$ such that for every $i,j\in R$ we have $\pi(A_i)= \pi(A_j)$, and for every $\{i,j\}\in E$ the pair $(A_i,A_j)$ is an $\alpha$-regular pair in $W$. Then
\begin{equation}
\label{eq:CountingLemma}
t^{\mathbf{A}}(F,W)= \left( \prod_{ij \in E}d(A_i, A_j) \pm \alpha |R|^{|R|} \right) \prod_{i\in R}\pi(A_i)\;.
\end{equation}
\end{lem}
\begin{proof}
We derive this from the version of the counting lemma stated as Lemma~10.24 in~\cite{Lovasz2012}, where we refer to the notion of decorated graphs and the corresponding homomorphism density. That counting lemma is tailored to the notion of cut norm distance, which is linked to the weak regularity lemma. On the other hand, Lemma~\ref{lem:RLcounting} is phrased in the setting of Szemer\'edi's regularity lemma, which is used much less in the theory of graph limits.\footnote{Actually, we are not aware of a similar application of Szemer\'edi's regularity lemma in the setting of graphons.} In particular, the error bound~\eqref{eq:CountingLemma} is relative to $\prod_{i\in R}\pi(A_i)$. So, the routine work below is to transfer to that setting, which we do so by zooming in on the sets $A_i$.

Further, assume that $R = [r]$ for some $r\in\N$. For $i\in [r]$, write $I_i := [\frac{i - 1}r,\frac{i}{r})$. Further, write $I=\bigcup_{i=1}^r I_i=[0,1)$. For every $i$ we define two probability measures: the $r$-times rescaled Lebesgue measure $\mu_i (\cdot)= r \mathrm{Leb}(\cdot)$ on $I_i$ and $\pi_i(\cdot) = \frac{\pi(\cdot)}{\pi(A_i) }$ on $A_i$. Since $(\Omega,\pi)$ is atomless and separable, the Isomorphism theorem for measure spaces tells us that there exists a measure-preserving map $f_i : (I_i, \mu_i) \to (A_i, \pi_i)$. For example, if $\Omega = ([0,1], \mathrm{Leb})$, we can define $f_i: I_i \rightarrow A_i$ as
$f_i(z) := \inf\{t\in A_i:\pi(A_i\cap[0,t])\ge (rz - (i - 1) )\cdot \pi(A_i)\}$.

Define $f : I \to \Omega$ by setting $f(z) = \sum_{i \in R} f_i(z) \indic_{z \in I_i}$.
Define $W^f:I^2\rightarrow[0,1]$ by setting
\[
  W^f(x,y) := W \left(f(x),f(y)\right)\;.
\]

Consider the complete graph $K$ on the vertex set $R$. We construct two decorated graphs $u = (U_e:e\in E(K))$ and $u' = (U'_e:e\in E(K))$ on the ground space $I$ as follows. For every $e = \{i,j\}\in E(K)$, writing $d_e = d(A_i, A_j)$, we define
\begin{align*}
  U_e&=
  \indic_{I_i \times I_j \cup I_j \times I_i} \cdot \begin{cases}
    W^f , &\quad \text{if  $e\in E$,} \\
    1, &\quad \text{if $e\in E(K)\setminus E$,}
  \end{cases}\\
  U'_e&=
  \indic_{I_i \times I_j \cup I_j \times I_i} \cdot \begin{cases}
    d_e, &\quad \text{if  $e\in E$,} \\
  1, &\quad \text{if $e\in E(K)\setminus E$.}
  \end{cases}
\end{align*}
It is obvious that for every $x_1,\ldots,x_r\in I$ we have
\begin{equation}
\label{eq:indicator_polka}
\prod_{ij \in E(K)}U_{ij}(x_i, x_j) = \prod_{i,j \in E} W^f(x_i, x_j )\prod_{i \in [r]} \indic_{I_i}(x_i).
\end{equation}
Similarly,
\begin{equation}
\label{eq:indicator_jive}
\prod_{ij \in E(K)}U'_{ij}(x_i, x_j) = \prod_{e \in E} d_e \prod_{i \in [r]} \indic_{I_i}(x_i).
\end{equation}
Note that $(x_1, \dots, x_r) \mapsto ( f(x_1), \dots f(x_r) )$ is a measure-preserving map between the product of probability spaces $(I_i, \mu_i)$ and the product of probability spaces $\left( A_i, \pi_i \right)$. Therefore by \eqref{eq:indicator_polka},
\[
  r^{r} t(K, u) = \int_{\prod_i I_i} \prod_{ij \in E} W^f(x_i, x_j) \prod_i \D \mu_i 
  = \int_{\prod_i A_i} \prod_{ij\in E} W(y_i, y_j) \prod_i \D \pi_i =  \frac{ t^\mathbf{A}(F,W)}{\prod_i\pi(A_i)}
\]
and, by \eqref{eq:indicator_jive},
\[
  r^{r} t(K, u') = \int_{\prod_i I_i} \prod_{e \in E} d_{e} \prod_i \D \mu_i = \prod_{e \in E} d_e .
\]
To sum up,
\begin{align}
\begin{split}
\label{eq:comparezwei}
t(K,u)& = t^{\mathbf{A}}(F,W)\cdot \frac{r^{-r}}{\prod_{i\in R}\pi(A_i)} \;\quad \mbox{and}\\
t(K,u')&=\prod_{e\in E}d_e\cdot r^{-r}\;.
\end{split}
\end{align}
We claim that for each $e\in E(K)$ we have $\cutn{U_e - U'_e} \le 2\alpha r^{-2}$. We shall postpone the proof of this to the end, and first show how it implies the asserted statement. By Lemma~10.24 in~\cite{Lovasz2012}, we have
\begin{equation*}
|t(K,u) - t(K,u')| \le \sum_{e\in E(K)}\cutn{U_e - U'_e} \le \binom{r}{2}2\alpha r^{-2}<\alpha\;.
\end{equation*}
Substituting the expressions from~\eqref{eq:comparezwei} and multiplying by $r^r\cdot \prod_{i\in R}\pi(A_i)$, we obtain
\[
  \left|
  t^{\mathbf{A}}(F,W)
  -
  \prod_{e\in E}d_e \cdot \prod_{i\in R}\pi(A_i)
  \right|
  \le \alpha \cdot r^r \cdot \prod_{i\in R}\pi(A_i) \;,
\]
as needed.

We now return to the deferred claim that for each $e = \{a,b\}\in E(K)$ we have $\cutn{U_e - U'_e} \le 2\alpha r^{-2}$. Note that for $e = ab \notin E$ we have $U_{\{a, b\}} = U'_{\{a, b\}}= \indic_{I_a \times I_b \cup I_b \times I_a}$ hence trivially $\cutn{U_e - U'_e} = 0$. It remains to consider the case $e \in E$. Suppose that $S,T\subset I$ are arbitrary as in~\eqref{eq:defcnd}. Let $C_a := S\cap I_a$, $C_b := S\cap I_b$, $D_a:= T\cap I_a$, and $D_b:= T\cap I_b$. Since $I_a, I_b$ are disjoint and $U_e, U'_e$ are zero outside of $I_a \times I_b \cup I_b \times I_a$, we have
\begin{align*}
  \left|\int_{S\times T}(U_e - U'_e)\D x_a \D x_b\right| & \le \left|\int_{C_a\times D_b}(U_e - U'_e) \D x_a \D x_b\right| + \left|\int_{C_b\times D_a}(U_e - U'_e) \D x_a \D x_b \right|\;.
\end{align*}
Since $(x_a, x_b) \mapsto (f(x_a), f(x_b))$ is a measure preserving map between $(I_a, \mu_a) \times (I_b, \mu_b)$ and $(A_a, \pi_a) \times (A_b, \pi_b)$, we have
\begin{align*}
  r^2 &\left| \int_{C_a\times D_b}(U_e - U'_e)\D x_a \D x_b \right| = \left| \int_{C_a\times D_b}(W^f - d_e)\D \mu_a \D \mu_b \right| \\
  &= \left| \int_{f(C_a) \times f(D_b)} (W - d_e) \D \pi_a \D \pi_b \right| \\
  &= \frac{1}{\pi(A_a) \pi(A_b)}\left| \int_{f(C_a) \times f(D_b)} W  \D \pi^2 - \pi(f(C_a))\pi(f(D_b))d_e \right| \leByRef{eq:IAmRegular} \alpha.
\end{align*}
To conclude,
\[
  \left|\int_{C_a\times D_b}(U_e - U'_e)\D\pi^2\right| \le \alpha r^{-2}\;.
\]
Together with the same bound for $ \left|\int_{C_b\times D_a}(U_e - U'_e)\D\pi^2\right|$ we get that $ \left|\int_{S\times T}(U_e - U'_e)\D\pi^2\right| \le 2\alpha r^{-2}$, as was needed.
\end{proof}

\subsubsection{Inequality between clique densities}
Here we shall prove an inequality which relates clique densities in graphons. This is a consequence of the Kruskal--Katona theorem \cite{MR0154827, MR0290982}. In fact, our proof uses a slightly weaker version due to Lov\'asz~\cite[Exercise 13.31(b)]{MR1265492}. While we need this inequality only for the case $r=2$ (which can be derived from \cite{erdos1962number}), we shall state and prove the general case for possible future use. For $m\in\N$ and a real number $x \ge m$ we write $\binom{x}{m} := \frac{x(x-1)\ldots(x-m+1)}{m!}$.
\begin{prop} \label{prop:KruskalKatona}
  Let $W\in \Gra$ be a graphon and $r, k\in\N$ satisfy $r < k$. Then we have
  \[
    t(K_r,W)\ge t(K_k,W)^{r/k}\;.
  \]
  In particular, for all $k \ge 3$ we have $\int W \D \pi^2 = t(K_2,W) \ge t(K_k,W)^{2/k}$.
\end{prop}
\begin{proof}
  Define $\alpha := t(K_k,W)^{1/k}$. Pick a sequence $(G_n)_{n \in \N}$ of graphs with $v(G_n) = n$ which converges to the graphon $W$. Let $(u_n)_{n \in \N}$ be a sequence of real numbers so that the number of $k$-cliques in $G_n$ is $\binom{u_n}{k}$, that is, $t(K_k, G_n) = {\binom{u_n}{k}}/{\binom{n}{k}}$; this implies $u_n \sim \alpha n$. 
Exercise 13.31 in Lov\'asz~\cite{MR1265492} (for $H$ consisting of the vertex sets of $k$-cliques) implies that the number of $r$-cliques in $G_n$ is at least $\binom{u_n}{r}$, whence
\[
  t(K_r, W) = \lim_{n \to \infty} t(K_r, G_n) \ge \lim_{n \to \infty} \frac{\binom{u_n}{r}}{\binom{n}{r}} = \alpha^r = t(K_k,W)^{r/k} \;,
\]
as required.
\end{proof}

\subsection{Probability}

The classical Berry--Esseen theorem establishes quantitative bounds on the accuracy of the approximation of the probability distribution of a sum of independent and identically distributed random variables by a normal distribution with the same mean and variance. We will apply this to prove Theorem~\ref{thm:extremist-threshold-upper}.

\begin{thm}[Berry--Esseen theorem] \label{thm:berry-esseen}
There is a positive constant $C$ such that the following holds. Suppose that $X_1,\dots,X_k$ are i.i.d.\ random variables with $\E X_i = 0$, $\var X_i = \sigma^2 < \infty$ and $\E|X_i|^3 = \beta < \infty$ for all $i\in[k]$. Let $X=\sum_{i=1}^{k}X_i$. Then for all $x\in\R$ we have
\[
  \left| \prob{X \le x\sigma\sqrt{k}} - \prob{Z \le x} \right| \le \frac{C\beta}{\sigma^3\sqrt{k}},
\]
where $Z\sim N(0,1)$ is a standard Gaussian random variable.
\end{thm}

\section{Flip processes as dynamical systems on graphons: a recap}

In this section, we review the theory from~\cite{Flip1}. With the kind permission of the authors of~\cite{Flip1}, we include the relevant parts into this section with varying degrees of additional edits.

\subsection{The velocity operator} \label{ss:velocity}

Let $W$ be a graphon and let $\rul$ be a rule of order $k$.
Given a graph $F \in \lgr{k}$ and an ordered pair $1 \le a \neq b \le k$ define an operator $\Troot{F}{a,b}:\Kernel \rightarrow L^\infty(\Omega^2)$ by setting
\[
(\Troot{F}{a,b} W) (x,y) = \tindr{(x,y)}(F^{a,b}, W)\;.
\]
The \emph{velocity operator for $\rul$} is the operator $\vel[\rul] :\Kernel \to \Kernel$ given by
\begin{equation} \label{eq:velocity}
\vel[\rul] W := \sum_{F,H\in\lgr{k}} \sum_{1 \le a\neq b \le k} \rul_{F,H} \cdot \Troot{F}{a,b}W \cdot \left(\ind{ab \in H \sm F} - \ind{ab \in F \sm H}\right)\;,
\end{equation}
where $ab$ is shorthand for $ \left\{ a,b \right\}$.

It is convenient to have alternative reformulations of the velocity operator. Let $W$ be a graphon and $\rul$ be a rule of order $k$. Let the random variables $(U_v : v \in [k])$ and the random graph $\G(k,W)$ be as defined in Section~\ref{ss:sampling}. Let $(\ff,\hh)$ be a pair of random graphs in $\lgr{k}$ such that $\ff \sim \G(k,W)$ and $\Pc{\hh = H}{\ff = F} = \rul_{F,H}$ for all $F,H\in\lgr{k}$. Then we have (see equation~\eqref{FIRSTPAPER.eq:vel_prob} in \cite{Flip1})
\begin{equation}
\label{eq:vel_prob}
\vel[\rul]W(x,y) = \sum_{a \neq b} \left[ \Pc{ ab \in \hh }{ U_a = x, U_b = y } - W(x,y) \right].
\end{equation}
It follows that
\begin{align}
\label{eq:vel_integral}
\int \vel[\rul] W \D \pi^2 &= 2 \E \left( e(\hh) - e(\ff) \right) \\
\label{eq:vel_integral2}
&= 2 \E e(\hh) - (k)_2 t(K_2, W).
\end{align}

\subsection{Trajectories}
In this subsection, we recap the notion of trajectories which we mentioned informally in Theorem~\ref{thm:correspondenceInformal}. We make several adjustments, following the way trajectories were constructed in~\cite{Flip1}. Firstly, we define trajectories not only for graphons but for general kernels. Secondly, we define not only forward evolution in time but also backward evolution. It turns out that in order to allow for these generalizations, we cannot prescribe the time domain for a trajectory in advance.
\begin{defi}
  \label{def:trajectory}
  Given a rule $\rul$ and a kernel $W \in \Kernel$, a \emph{trajectory starting at $W$} is a differentiable function $\traj{\cdot}{W} : I \to (\Kernel, \Linf{\cdot})$ defined on an open interval $I \subseteq \R$ containing zero, that satisfies the autonomous differential equation
\begin{equation}
\label{eq:PhiDeriv}
\frac{\D}{\D t}\traj{t}{W} = \vel \traj{t}{W}\;
\end{equation}
with the initial condition
\begin{equation}
\label{eq:Phi_initial}
\traj{0}{W} = W\;.\\
\end{equation}
\end{defi}
The following theorem implies that if a trajectory starts a graphon (rather than a general kernel), it is unique (provided we extended it maximally) and is defined for all positive times. Such a trajectory is always defined for some negative times (as $\mdom{W}$ is open), but it might leave the graphon space immediately (i.e., we might have $\life{W} = [0, \infty)$).
\begin{thm}[{Theorem~\ref{FIRSTPAPER.thm:flow} in \cite{Flip1}}]
  \label{thm:flow}
  Suppose that $\rul$ is a rule and $W \in \Kernel$ is a kernel. Then the following hold.
\begin{romenumerate}
\item \label{en:exist}
  There is an open interval $\mdom{W} \subseteq \R$ containing $0$ and a trajectory $\traj{\cdot}{W} : \mdom{W} \to \Kernel$ starting at $W$ such that any other trajectory starting at $W$ is a restriction of $\traj{\cdot}W$ to a subinterval of $\mdom{W}$.

\item \label{en:semi}
  For any $u \in \mdom{W}$ we have $\mdom{\traj{u}{W}} = \{ t \in \R : t + u \in \mdom{W}\}$ and for every $t \in \mdom{\traj{u}{W}}$ we have
  \begin{equation*}
    \traj{t}{ \traj{u}{W} } = \traj{t+u}{W}.
  \end{equation*}

\item \label{en:life}
  If $W \in \Gra$ is a graphon, the set $\life{W} := \{ t \in \mdom{W} : \traj{t}W \in \Gra\}$ is a closed interval containing $[0, \infty)$.
\end{romenumerate}

\end{thm}

\begin{defi}
  For $W\in\Gra$ we write $\age{W} := -\inf \life{W}\in [0, \infty]$\index{$\age(W)$}.
\end{defi}
\begin{remark}[{Remark~\ref{FIRSTPAPER.rem:integral_form} in \cite{Flip1}}]
  \label{rem:integral_form}
  A trajectory $\traj{\cdot}W : I \to \Kernel$ satisfies
  \begin{equation*}
    \traj{t} W = W + \int_0^t \vel[\rul] \traj{\tau} W \D \tau, \quad t \in I,
  \end{equation*}
where the integral is defined with respect to the norm $\Linf{\cdot}$.
\end{remark}

The last preliminary result we shall need about trajectories says the following. Let $\rul$ be a rule and $\ca$ be a set of graphons which have a certain profile on a fixed rectangle $\Omega_1\times \Omega_2$. Suppose that for each graphon in $\ca$, the velocity restricted to the rectangle $\Omega_1\times \Omega_2$ is zero. Then each trajectory which starts in $\ca$ stays confined to $\ca$.
\begin{prop}[Proposition~\ref{FIRSTPAPER.prop:zeroonasection} in \cite{Flip1}]\label{prop:zeroonasection}
  Let $\rul$ be a rule and $\alpha:S\rightarrow[0,1]$ be a symmetric function where $S := \Omega_1\times\Omega_2\cup \Omega_2\times\Omega_1 \subset \Omega^2$. Let $\ca=\{W\in\Gra:W{\restriction_S}=\alpha\}$. Suppose that for each $W\in \ca$ we have $(\vel W){\restriction_S}=0$. Then for each $U\in \ca$ and each $t\ge 0$ we have $\traj{t}{U}\in \ca$.
\end{prop}

\subsection{Continuity properties of trajectories}
In this subsection we assume that $\rul$ is a rule of order $k$.

\begin{lem}[{Lemma~\ref{FIRSTPAPER.lem:contderLinfty} in~\cite{Flip1}}] \label{lem:velocity-bound}
For every $W\in \Gra$ we have
\begin{equation} \label{eq:vel_bound_by_W}
  -(k)_2 \cdot W \le \vel[\rul] W \le (k)_2\cdot (1 - W)\;.
\end{equation}
\end{lem}

\begin{lem}[{Lemma~\ref{FIRSTPAPER.lem:LinftyCont} in~\cite{Flip1}}] \label{lem:traj-cont-Linfty}
Write $C_k = 2 \binom{k}{2}^2 2^{\binom{k}{2}}$. For every $W \in \Gra$ and every $t,u \in [-\age{W},+\infty)$ we have
\begin{equation*}
\Linf{\traj{t}{W} - \traj{u}{W}} \le (k)_2 |t - u|,
\end{equation*}
and for every $\delta > 0$ we have
\begin{equation}\label{eq:LinfPhi_eps}
  \Linf{\frac{\traj{\delta}{W} - W}{\delta} - \vel[\rul]W} \le C_k(k)_2 \delta.
\end{equation}
\end{lem}

\begin{thm}[{Theorem~\ref{FIRSTPAPER.thm:genome} in \cite{Flip1}}] \label{thm:LipschTime}
  Write $C_{\square, k} = 4\bik^2 2^{\binom{k}{2}}$.  Then for each $U,W \in \Gra$ and $t \in [-\min(\age{U},\age{W}),+\infty)$, we have
	\[
	  \cutn{\traj{t}{U} - \traj{t}{W}} \le \exp(C_{\square, k} |t|)\cdot \cutn{U - W}\;.
	\]
\end{thm}

\subsection{Destinations and speed of convergence}
Let $\rul$ be a rule. We say that the trajectory of a graphon $U$ is \emph{convergent} if there exists a \emph{destination} of $U$, that is, a graphon denoted by $\dest_\rul(U)$ such that $\traj{t}{U} \tocutn \dest_\rul(U)$ as $t \to \infty$. We say that the flip process with rule $\rul$ is \emph{convergent} if all its trajectories converge.
\begin{prop}[{\cite[Proposition~\ref{FIRSTPAPER.prop:destinations}]{Flip1}}]\label{prop:dests}
  If $W = \dest_{\rul}(U)$ for some $U \in \Gra$, then $\vel[\rul] W~=~0$.
\end{prop}

For a rule $\rul$, and $\delta>0$ we define the \emph{lower convergence time} $\tau_{\rul}^-(\delta)$ and the \emph{upper convergence time} $\tau_{\rul}^+(\delta)$ by
\begin{align*}
\tau_{\rul}^-(\delta) &:= \sup \{\inf\{t : \cutnd(\traj{t}{W},\dest_{\rul}(W))<\delta \}:W\in \Gra\}\;,\\
\tau_{\rul}^+(\delta) &:= \sup \{\sup\{t : \cutnd(\traj{t}{W},\dest_{\rul}(W))\ge \delta \}:W\in \Gra\}\;.
\end{align*}

\section{Ignorant flip processes}\label{sec:ignorant}
The first and in a sense the most trivial class of flip processes we introduce are the ignorant flip processes. These are flip processes in which the replacement graph is selected independently of the drawn graph.

\begin{defi}
Let $k\in\N$. We say that a rule $\rul$ of order $k$ is \emph{ignorant} if for any $F,F',H\in\lgr{k}$ we have $\rul_{F,H} =\rul_{F',H}$. The \emph{ignorant output density} is $d := \frac{1}{\binom{k}{2}} \sum_{H\in\lgr{k}}\rul_{F,H}e(H)$ (here $F$ is arbitrary).
\end{defi}
Observe that the Erd\H{o}s--R\'enyi flip process is an ignorant flip process of order~2 with output distribution $\{\rul_{F,H}\}_{H\in\lgr{2}} = \Dirac_{K_2}$ for each $F\in\lgr{2}$.

Our main result about ignorant rules is Proposition~\ref{prop:ignorant-behaviour}, which explicitly describes their trajectories, destinations, speed of convergence and age.

\begin{prop} \label{prop:ignorant-behaviour}
  Let $k\in\N$. Suppose that $\rul$ is an ignorant rule of order $k$ with ignorant output density $d$. Then for each graphon $W \in \Gra$ and $t\in \life{W}$ we have
\begin{equation} \label{eq:ignoranttrajectory}
\traj{t}W = d + ( W - d ) \cdot e^{-(k)_2t}.
\end{equation}
Hence, we have $\dest(W) \equiv d$ and
\begin{equation} \label{eq:ignorant-convergence}
\tau_{\rul}^-(\delta)=\tau_{\rul}^+(\delta)=\frac{1}{(k)_2}\ln\left(\frac{\max\{d,1-d\}}{\delta}\right)
\end{equation}
for each $0 < \delta < \frac{1}{2}$. Furthermore, writing $i_W := \essinf(W)$ and $s_W := \esssup(W)$, we have
\begin{equation}\label{eq:ageignorant}
  \age{W} =
  \begin{cases}
    \frac{1}{(k)_2} \min \left\{\ln \frac{d}{d-i_W} \;,\; \ln \frac{1-d}{s_W-d}\right\} \quad & \text{if } i_W < d < s_W\;, \\
    \frac{1}{(k)_2} \ln \frac{d}{d-i_W} \quad &\text{if } i_W < s_W \le d\;, \\
    \frac{1}{(k)_2} \ln \frac{1-d}{s_W-d} \quad &\text{if } d \le i_W < s_W\;, \\
    \infty & \text{ if } W \equiv d.
  \end{cases}
\end{equation}
In particular, $\age{W}$ is finite unless $W \equiv d$.
\end{prop}

\begin{proof}
  Let $W\in\Gra$. We first compute the velocity $\vel[\rul]$ using the notation preceding equation~\eqref{eq:vel_prob}. Since the distribution of $\hh$ is independent of $(U_v, v \in [k])$, by~\eqref{eq:vel_prob} we have
\begin{align*}
\vel[\rul]W(x,y) &= \sum_{a \neq b} \left[ \Pc{ ab \in \hh }{ U_a = x, U_b = y } - W(x,y) \right] \\ &= 2 \E e(\hh) - (k)_2 W(x,y) = (k)_2 \left( d - W(x,y) \right).
\end{align*}

It is clear that the initial condition $\traj{t}{W} = W$ is satisfied.	
By the uniqueness of trajectories, it is enough to verify that the function on the right hand side of \eqref{eq:ignoranttrajectory} satisfies the differential equation \eqref{eq:PhiDeriv}. For $x,y \in \Omega$ we set $f(t) = d + (W(x,y) - d) \cdot e^{-(k)_2t}$. Observe that for each $t \ge 0$ we have
\[
  \frac{f(t + \eps) - f(t)}{\eps} \to f'(t) = (k)_2 \left(d - W(x,y)\right) e^{-(k)_2t} = (k)_2 \left(d - f(t)\right)
\]
as $\eps \to 0$, where the convergence is uniform over all choices of $(x,y)\in\Omega^2$ by Taylor's theorem because $W$ is bounded. This implies that~\eqref{eq:ignoranttrajectory} indeed describes the trajectory starting at $W$ under the ignorant rule $\rul$.

Clearly, all trajectories converge to the same destination, the constant-$d$ graphon. By~\eqref{eq:ignoranttrajectory} we have $\cutnd(\traj{t}{W},d)\le \Linf{\traj{t}{W}-d} \le \max\{d,1-d\} e^{-(k)_2t}$, with both equalities attained at the constant-$0$ graphon if $d \ge 1/2$ and at the constant-$1$ graphon if $d \le 1/2$. Hence, we have $\tau_{\rul}^-(\delta)=\tau_{\rul}^+(\delta)=\frac{1}{(k)_2}\ln\left(\frac{\max\{d,1-d\}}{\delta}\right)$ for each $\delta\in(0,\frac{1}{2})$. Now observe that~\eqref{eq:ignoranttrajectory} tells us that for a point $(x,y)$ where $W(x,y)<d$, the value of $\traj{-t}{W}$ dips below~$0$ at $t = \frac{1}{(k)_2}\ln\frac{d}{d-W(x,y)}$, and for a point $(x,y)$ where $W(x,y)>d$, the value of $\traj{-t}{W}$ exceeds~$1$ at $t = \frac{1}{(k)_2}\ln\frac{1-d}{W(x,y)-d}$. This implies~\eqref{eq:ageignorant}.
\end{proof}

\section{Stirring flip processes}\label{sec:stirring}
All trajectories of ignorant flip processes converge to a constant given by the ignorant output density. Being a bit less ignorant this time, we introduce \emph{stirring flip processes}, which takes into account the edge count (but nothing else) of the drawn graph. In this section we focus on \emph{balanced} stirring flip processes, in which the expected number of edges in the replacement graph is equal to the number of edges in the drawn graph. In Section~\ref{sec:extremist} we will study \emph{extremist flip processes}, which are ``very unbalanced'' stirring flip processes.
\begin{defi}
Let $k\in\N$. We say that a rule $\rul$ of order $k$ is \emph{stirring} if for each $F$ and $F'$ with $e(F)=e(F')$ and each $H$, we have $\rul_{F,H} =\rul_{F',H}$. We say that the rule $\rul$ is \emph{balanced} if for each $F\in\lgr{k}$ we have $e(F) = \sum_{H\in\lgr{k}}\rul_{F,H}e(H)$.
\end{defi}
Two natural examples of a balanced stirring rule are the rule where a drawn graph with $\ell$ edges is replaced by the Erd\H{o}s--R\'enyi uniform random graph with $\ell$ edges and the rule where a drawn graph with $\ell$ edges is replaced by the Erd\H{o}s--R\'enyi binomial random graph with edge probability $\ell/\binom{k}{2}$.

Our main result describes the trajectories of a stirring flip process in the most important balanced case.
\begin{prop} \label{prop:stirring-trajectory}
  Suppose that $\rul$ is a balanced stirring rule of order $k \ge 3$ and let $W \in \Gra$. Then for every $t\in \life{W}$ we have
\begin{equation} \label{eq:stirring-trajectory}
\begin{split}
\traj{t}{W}(x,y) &= \Lone{W} + (\deg_W(x) + \deg_W(y) - 2\Lone{W}) \cdot e^{-t(k-1)_2} \\
&\quad + (W(x,y) - \deg_W(x) - \deg_W(y) + \Lone{W}) \cdot e^{-t((k)_2-2)}\;.
\end{split}
\end{equation}
In particular, the edge density stays constant along the trajectory and $\dest(W) = \Lone{W}$.
\end{prop}
\begin{proof}
Let $W \in \Gra$ and set $d := \Lone{W}$. The following is our key claim.

\begin{claim} \label{claim:stirring-velocity}
Given a graphon $W\in\Gra$, for all $x,y\in\Omega$ we have
\begin{equation} \label{eq:stirring-velocity}
\vel[\rul]W(x,y) = (2k-4)(\deg_W(x) + \deg_W(y)) + (k-2)_2 d - \left( (k)_2 - 2 \right) W(x,y).
\end{equation}
\end{claim}

Before we prove this claim, let us first demonstrate how it implies our theorem. By the uniqueness of trajectories, it is enough to verify that \eqref{eq:stirring-trajectory} satisfies the differential equation \eqref{eq:PhiDeriv}. Write $f_{(x,y)}(t)$ for the right-hand side of \eqref{eq:stirring-trajectory}, set $U_t$ to be the graphon given by $U_t(x,y) := f_{(x,y)}(t)$ and set $g_x(t) := \deg_{U_t}(x) = d + (\deg_W(x) - d)e^{-t(k-1)_2}$. Observe that
\begin{align*}
f'_{(x,y)}(t) &= -(k-1)_2(\deg_W(x)+\deg_W(y)-2d)e^{-t(k-1)_2} \\
&\;\quad - ((k)_2-2)(W(x,y)-\deg_W(x)-\deg_W(y)+d)e^{-t((k)_2-2)} \\
&= \left( (k)_2 - 2 \right) (d - f_{(x,y)}(t)) + 2(k-2)(\deg_W(x) + \deg_W(y) - 2d)e^{-t(k-1)_2} \\
&= (2k-4)(g_x(t) + g_y(t)) + (k-2)_2 d - \left( (k)_2 - 2 \right) f_{(x,y)}(t)\;.
\end{align*}
Furthermore, for each $t \ge 0$ the convergence in $(f_{(x,y)}(t + \eps) - f_{(x,y)}(t))/\eps \to f'_{(x,y)}(t)$ is uniform over all choices of $(x,y)\in\Omega^2$ by Taylor's theorem because $W$ is bounded. This implies that $\frac{\D}{\D t} U_t = \vel[\rul] U_t$ in the norm $\Linf{\cdot}$. Moreover, we have that $U_0 = W$. Hence, by uniqueness \eqref{eq:stirring-trajectory} indeed describes the trajectory starting at $W$ under a balanced stirring rule.

It remains to prove Claim~\ref{claim:stirring-velocity}.

\begin{claimproof}[Proof of Claim~\ref{claim:stirring-velocity}]
  Let $W\in\Gra$ and $x,y\in\Omega$. Using the notation preceding equation~\eqref{eq:vel_prob}, observe that
\begin{equation} \label{eq:stirring-cond-edges-F}
\begin{split}
\Ec{e(\ff)}{U_1 = x, U_2 = y} &= \frac{1}{2}\sum_{a \neq b} \Pc{ ab \in \ff }{ U_1 = x, U_2 = y } \\
&=W(x,y) + (k-2) \left(\deg_W(x) + \deg_W(y)\right) + \binom{k-2}{2}d.
\end{split}
\end{equation}
Setting $q_m(ab) := \Pc{ab \in \hh}{e(\ff) = m}$, we have
\begin{equation} \label{eq:cond-edge-prob-ab}
\sum_{a \neq b}  q_m(ab) = 2\Ec {e(\hh)}{e(\ff) = m} = 2m.
\end{equation}
Since the distribution of $\ff$ is invariant under permutation of vertices, for any pair of distinct vertices $a$ and $b$ the quantity $\Pc{e(\ff) = m}{U_a = x, U_b = y}$ is the same; we shall denote it by $p_{x,y}(m)$. We have
\begin{equation} \label{eq:cond-edges-ab}
\Pc{ ab \in \hh }{ U_a = x, U_b = y } = \sum_{m = 0}^{\bik} p_{x,y}(m)q_m(ab).
\end{equation}
By summing~\eqref{eq:cond-edges-ab} over all pairs $a,b$ of distinct vertices, interchanging the order of summation, and applying~\eqref{eq:cond-edge-prob-ab}, we obtain
\begin{align*}
\sum_{a \neq b} \Pc{ ab \in \hh }{ U_a = x, U_b = y } = \sum_{m = 0}^{\bik} p_{x,y}(m) \sum_{a \neq b}  q_m(ab) &= \sum_{m = 0}^{\bik} \left[ p_{x,y}(m) \cdot 2 m \right] \\
&= 2\Ec{e(\ff)}{U_1 = x, U_2 = y}.
\end{align*}
Now by recalling \eqref{eq:vel_prob} and applying~\eqref{eq:stirring-cond-edges-F}, we obtain
\begin{align*}
\vel[\rul]W(x,y) &= \sum_{a \neq b} \left[ \Pc{ ab \in \hh }{ U_a = x, U_b = y } - W(x,y) \right], \\
&= 2\Ec{e(\ff)}{U_1 = x, U_2 = y} - (k)_2W(x,y) \\
&= (2k-4)(\deg_W(x) + \deg_W(y)) + (k-2)_2 d - W(x,y) \left( (k)_2 - 2 \right)
\end{align*}
as required.
\end{claimproof}
This completes the proof of the proposition.
\end{proof}

\subsection{Age of a balanced stirring process}
Recall that the age of a graphon $W$ is the infimum of times $t\ge 0$ for which the range of $\traj{-t}{W}$ is not wholly contained in $[0,1]$. That is, given a balanced stirring rule as in Proposition~\ref{prop:stirring-trajectory}, we want to know the infimum of $t\ge0$ where we have for some $x,y\in\Omega$ that
\begin{equation} \label{eq:Iwlos}
\begin{split}
\Lone{W} &+ (\deg_W(x) + \deg_W(y) - 2\Lone{W}) \cdot e^{t(k-1)_2} + \\
&+ (W(x,y) - \deg_W(x) - \deg_W(y) + \Lone{W}) \cdot e^{t((k)_2-2)}\;
\end{split}
\end{equation}
is more than~1 or less than~0. While there is no simple formula to deduce the age from~\eqref{eq:Iwlos} in general because we have two exponential terms with different coefficients and exponents, in the next proposition we give a formula for the age in two special cases: degree-regular graphons and graphons with a separable form. These correspond to the cases where only one exponential term has a positive coefficient. Further, we show that the age is finite except when the graphon is constant.

\begin{prop}\label{prop:agestirring}
Suppose that $\rul$ is a balanced stirring rule of order $k\ge3$. Then the following hold for all $W\in \Gra$. Write $d := \Lone{W}$. If $W = d$, then $\age{W} = \infty$. Otherwise, we have $\age{W} < \infty$ and in particular:
\begin{romenumerate}
\item\label{en:FF1} if $W$ is degree-regular, then
\begin{equation*}
  \age{W} = \frac{1}{(k)_2-2} \cdot \min \left\{\ln \frac{d}{d-\essinf(W)} \;,\; \ln \frac{1-d}{\esssup(W)-d}\right\}\;;
\end{equation*}
\item\label{en:FF1b} if $W$ is of the form $W(x,y) = f(x) + f(y)$, then
\begin{equation*}
  \age{W} = \frac{1}{(k-1)_1} \cdot \min \left\{\ln \frac{d}{d-\essinf(W)} \;,\; \ln \frac{1-d}{\esssup(W)-d}\right\}\;.
\end{equation*}
\end{romenumerate}
\end{prop}
\begin{proof}
  $W = d$ is clearly a fixed point, so we have $\age{W} = \infty$. We now show that $\age{W} = \infty$ only if $W = d$. Suppose that $\age{W} = \infty$. 
  Since $(k-1)_2 \neq (k)_2 - 2$, the expression in~\eqref{eq:Iwlos} can stay in $[0,1]$ for all $x,y \in\Omega$ only if the coefficients of both exponential terms are zero. 
  The first one implies $\deg_W \equiv d$ and then the second one implies $W \equiv d$.

  Hence, we further assume that $W$ is not constant, which implies $\essinf W < d < \esssup W$. In the case~\ref{en:FF1}, the equation~\eqref{eq:Iwlos} simplifies to
\begin{equation*}
d+(W(x,y) - d) \cdot e^{t((k)_2-2)}\;,
\end{equation*}
and then the formula for age indeed expresses the time when the values of $\traj{-t}{W}$ first cross one of the two thresholds $0$ and $1$.

  The case~\ref{en:FF1b} is equivalent to $W(x,y) = \deg_W(x) + \deg_W(y) - d$. Hence \eqref{eq:Iwlos} simplifies to
\begin{equation*}
d+(W(x,y) - d) \cdot e^{t(k-1)_1}\;,
\end{equation*}
which implies the age in the same way as in the case~\ref{en:FF1}.
\end{proof}

\section{Complementing flip processes} \label{sec:complementing}

The \emph{complementing rule of order $k$} is the rule of order $k$ where the replacement graph is the complement of the drawn graph. Our main result about complementing rules is Proposition~\ref{prop:complementing-behaviour}, which explicitly describes their trajectories, destinations, speed of convergence and age.

\begin{prop} \label{prop:complementing-behaviour}
  Suppose that $\rul$ is the complementing rule of order $k\ge 2$. Then for each graphon $W\in\Gra$ and $t\in \life{W}$ we have
\begin{equation} \label{eq:complementing-trajectory}
\traj{t}{W}(x,y) = \frac{1}{2} + \left( W(x,y) - \frac{1}{2} \right) \cdot e^{-2(k)_2t}\;.
\end{equation}
Hence, we have $\dest(W) \equiv 1/2$ and 
\begin{equation} \label{eq:complementing-convergence}
\tau_{\rul}^-(\delta)=\tau_{\rul}^+(\delta)=\frac{-\ln(2\delta)}{2(k)_2}
\end{equation}
for each $0 < \delta < \frac{1}{2}$. Furthermore, writing $i_W := \essinf(W)$ and $s_W := \esssup(W)$, we have
\begin{equation}\label{eq:complementing-age}
  \age{W} =
  \begin{cases}
    \frac{1}{2(k)_2} \min \left\{\ln \frac{1}{1-2i_W} \;,\; \ln \frac{1}{2s_W-1}\right\} \quad & \text{if } i_W < 1/2 < s_W\;, \\
    \frac{1}{2(k)_2} \ln \frac{1}{1-2i_W} \quad &\text{if } i_W < s_W \le 1/2\;, \\
    \frac{1}{2(k)_2} \ln \frac{1}{2s_W} \quad &\text{if } 1/2 \le i_W < s_W\;, \\
    \infty & \text{ if } W \equiv 1/2.
  \end{cases}
\end{equation}
\end{prop}
\begin{proof}
  We calculate the velocity using the notation preceding the formula~\eqref{eq:vel_prob}. Note that $ab \in \hh$ if and only if $ab \notin \ff$, so
\[
  \Pc{ ab \in \hh }{ U_a = x, U_b = y } = \Pc{ ab \notin \ff }{ U_a = x, U_b = y } = 1 - W(x,y)\;.
\]
Hence, by~\eqref{eq:vel_prob} we have
\[
  \vel[\rul]W(x,y) = \sum_{1 \le a \neq b \le k} (1 - W(x,y) - W(x,y)) = (k)_2 (1 - 2W(x,y))\;.
\]
Write $f_{(x,y)}(t)$ for the right-hand side of~\eqref{eq:complementing-trajectory}. Observe that for each $t \ge 0$ we have
\[
  \frac{f_{(x,y)}(t + \eps) - f_{(x,y)}(t)}{\eps} \to f'_{(x,y)}(t) = (k)_2(1-2W(x,y))e^{-2(k)_2t} = (k)_2 \left(1-2f_{(x,y)}(t)\right)
\]
as $\eps \to 0$, where the convergence is uniform over all choices of $(x,y)\in\Omega^2$ by Taylor's theorem because $W$ is bounded. Since the trajectory $\traj{\cdot}W$ is unique by Theorem~\ref{thm:flow}, we have $\traj{t}{W}(x,y)=f_{(x,y)}(t)$ as desired.

Clearly, the unique destination is the constant-$\frac{1}{2}$ graphon. By~\eqref{eq:ignoranttrajectory} we have $\cutnd(\traj{t}{W},\frac{1}{2})\le \Linf{\traj{t}{W}-\frac{1}{2}} \le \frac{1}{2} e^{-2(k)_2t}$, with both equalities attained at the constant-$0$ and constant-$1$ graphons. Hence, we have $\tau_{\rul}^-(\delta)=\tau_{\rul}^+(\delta)=\frac{-\ln(2\delta)}{2(k)_2}$ for each $\delta\in(0,\frac{1}{2})$. Note that~\eqref{eq:complementing-trajectory} tells us that for a point $(x,y)$ where $W(x,y)<\frac{1}{2}$, the value of $\traj{-t}{W}$ dips below~$0$ at $t = -\frac{1}{2(k)_2}\ln(1-2W(x,y))$, and for a point $(x,y)$ where $W(x,y)>\frac{1}{2}$, the value of $\traj{-t}{W}$ exceeds~$1$ at $t = -\frac{1}{2(k)_2}\ln(2W(x,y)-1)$. This implies~\eqref{eq:complementing-age}.
\end{proof}

We conjecture that the complementing flip process exhibits the fastest convergence among flip processes whose trajectories converge to a single destination. Observe that the convergence time of the complementing flip process (see~\eqref{eq:complementing-convergence}) is half that of an ignorant process with output density $1/2$ (cf.~\eqref{eq:ignorant-convergence}).

\begin{conj} \label{conj:complementingfastest}
Suppose that $\rul$ is a convergent flip process of order $k$ with a single destination. Then for each $\delta\in(0,\frac{1}{2})$ we have $\tau_{\rul}^-(\delta)\ge \frac{-\ln(2\delta)}{2(k)_2}$.
\end{conj}

\section{Extremist flip processes} \label{sec:extremist}

In this section we define the extremist rules and prove results about their behaviour.

\begin{defi}[Extremist rule]
Let $k\in\N$. The \emph{extremist} rule $\rul$ of order $k$ is defined by
\begin{equation} \label{eq:extremist-rule}
\rul_{F,H}=
\begin{cases}
1&\textrm{ if }e(F)<\frac{1}{2}\binom{k}{2}\textrm{ and }H=\overline{K_k}, \\
1&\textrm{ if }e(F)>\frac{1}{2}\binom{k}{2}\textrm{ and }H=K_k, \\
1&\textrm{ if }e(F)=\frac{1}{2}\binom{k}{2}\textrm{ and }H=F, \\
0&\textrm{ otherwise }.
\end{cases}
\end{equation}
\end{defi}

A possible variation of the extremist rule is to permit the replacement of graphs with $\binom{k}{2}/2$ edges by a clique or empty graph with equal probabilities. In fact, most proofs from this section (except for Theorem~\ref{thm:extremist-zero-threshold-bound}) work for this variation, but for the sake of simplicity we shall not consider it.

\subsection{Behaviour on constants}

We begin by considering the behaviour of trajectories under extremist rules on constant graphons. Recall that a trajectory that starts with a constant graphon takes values only among constant graphons and the velocity is a constant function (see \cite[Corollary~\ref{FIRSTPAPER.coro:steps_stay}]{Flip1}). Therefore, we shall identify the constant-$p$ kernel with the number $p$. We prove the following proposition, which implies that the trajectory of an extremist rule which starts from the constant-$p$ graphon converges to $0$ if $p< 1/2$, to $1$ if $p > 1/2$ and is fixed for $p = 1/2$.

\begin{prop} \label{prop:extremist-constant-velocity}
If $\rul$ is the extremist rule of order $k \ge 3$, then we have $\vel[\rul] (p) = 0$ for $p \in \left\{ 0, 1/2, 1 \right\}$, $\vel[\rul] (p) < 0$ for $p\in (0,1/2)$ and $\vel[\rul] (p) > 0$ for $p\in (1/2,1)$.
\end{prop}
\begin{proof}
  Since $\vel[\rul](p)$ is a constant kernel, we have $\vel[\rul](p) \equiv \int_{\Omega^2} \vel[\rul](p)$. Therefore, by~\eqref{eq:vel_integral2} and denoting by $\G$ the random graph $\G(k,p)$, we have
\begin{align*}
  \vel[\rul] (p) 	&=    2 \left[ \bik \Prob \left(e(\G) > \bik/2\right) + \bik/2 \cdot \Prob \left(e(\G) = \bik/2\right)  \right] - (k)_2 p.
\end{align*}
Writing $n = \binom{k}{2}$ and $B_{n,p}$ for a binomial random variable with parameters $n$ and $p$, we have
\begin{align*}
	\vel[\rul] (p) &=  2n \Prob(B_{n,p} > n/2) + n \Prob(B_{n,p} = n/2)  - 2np \\
	&= n \big( \Prob(B_{n,p} > n/2) + \Prob(B_{n,p} \ge n/2) - 2p \big).
\end{align*}
It is clear that $\vel[\rul] (p) = 0$ for $p \in \left\{ 0,1/2,1 \right\}$. To prove the proposition, it is enough to show that for the function
\[
  f_n(p) := \Prob(B_{n,p} > n/2) + \Prob(B_{n,p} \ge n/2)
\]
its derivative $f_n'$ is strictly increasing on $(0,1/2)$ and strictly decreasing on $(1/2, 1)$. Indeed, this implies that $p \mapsto f_n(p)$ is strictly convex on $(0,1/2)$ and strictly concave on $(1/2, 1)$. Then, these properties are inherited by $p \mapsto \vel[\rul](p)$, thereby implying that $\vel[\rul] (p)$ is negative on $(0,1/2)$ and positive on $(1/2,1)$. Observe that
\begin{align*}
	\frac{\D}{\D p} \left( \Prob(B_{n,p} = k) \right) &= \frac{\D}{\D p} \left( \binom{n}{k}p^k(1-p)^{n-k} \right) \\
	&= \binom{n}{k} \left(  kp^{k-1}(1-p)^{n-k} - (n-k)p^k(1-p)^{n-k-1} \right) \\
	&= n \left[ \binom{n-1}{k-1}  p^{k-1}(1-p)^{n-k} - \binom{n-1}{k}p^k(1-p)^{n-k-1} \right] \\
	&= n \left[ \Prob(B_{n-1,p} = k - 1) - \Prob(B_{n-1,p} = k ) \right].
	\end{align*}
	Therefore, we have
	\begin{align*}
	f_n'(p)
	&= n\sum_{k > n/2} \left[ \Prob(B_{n-1,p} = k - 1) - \Prob(B_{n-1,p} = k) \right] \\
	&\quad + n\sum_{k \ge n/2} \left[ \Prob(B_{n-1,p} = k - 1) - \Prob(B_{n-1,p} = k) \right] \\
	&= n \cdot \left( \prob{B_{n-1,p} = \flo{\frac{n-1}{2}}} + \prob{B_{n-1,p} = \cei{\frac{n-1}{2}}} \right) \\
	&= n \cdot \begin{cases}
	2\Prob(B_{n-1,p} = \flo{\frac{n-1}{2}}) , &\quad n \text{ odd} \\
	\Prob(B_{n-1,p} = \flo{\frac{n-1}{2}}) \left(1+\frac{p}{1-p} \right) &\quad n \text{ even}
	\end{cases} \\
	&= \left( 1 + \indic_{n \text{ odd}} \right) n \cdot \binom{n-1}{\flo{\frac{n-1}{2}}} (p(1-p))^{\flo{\frac{n-1}{2}}},
\end{align*}
Now $k \ge 3$ implies that $\flo{\frac{n-1}{2}} \ge 1$, so the polynomial $f'_n(p)$ inherits from the parabola $p(1-p)$ the property of being strictly increasing on $(0,1/2)$ and strictly decreasing on $(1/2,1)$. This completes our proof.
\end{proof}

\subsection{Density thresholds}
In the previous subsection we saw that for constant graphons we have $\dest_\rul(W) = 0$ as soon as the density of $W$ is below $1/2$. In this subsection we consider when we start with an arbitrary initial graphon $W$ and investigate what density guarantees $\dest_\rul(W) = 0$. We define the \emph{density threshold} for the extremist rule $\rul$ of order $k\ge3$ to be
\[
  \theta(k) := \sup\{p \in [0,1] : \forall W \in \Gra \ \Lone{W} \le p \text{ implies } \dest_\rul(W) = 0 \}.
\]

In the following proposition, we collect a few useful properties related to the density threshold.

\begin{prop}
  \label{prop:extremist_symm}
  The following holds for the extremist rule $\rul$ of order $k\ge3$.
  \begin{romenumerate}
    \item \label{item:extremist_sym} For every $W\in\Gra$ and every $t \in \life{W}$ we have $\traj{t}(1-W) = 1 - \traj{t}W$. In particular, if $\dest_{\rul}(W) = 0$, then $\dest_{\rul}(1-W) = 1$.
  \item \label{item:extremist_sym_upper} For each $W\in\Gra$ satisfying $\Lone{W} > 1 - \theta(k)$ we have $\dest_{\rul}(W) = 1$.
  \item \label{item:extremist_sym_half} $\theta(k) \le 1/2$.
  \end{romenumerate}
\end{prop}
\begin{proof}
  First we show \ref{item:extremist_sym}.
  From formula \eqref{eq:vel_prob} it is easy to see that for any graphon $W$ we have $\vel(1-W)=-\vel W$, so $\frac{\D}{\D t}(1-\traj{t}W)=-\vel\traj{t}W=\vel(1 - \traj{t}(W))$, that is, $t \mapsto 1 - \traj{t} W$ is a trajectory. Since it starts at the same graphon as the trajectory $t \mapsto \traj{t}(1 - W)$, by uniqueness (see Theorem~\ref{thm:flow}) it follows that $\traj{t}(1-W) = 1 - \traj{t}W$ for every $t \in \life{W}$. The second conclusion is obvious.

  To prove~\ref{item:extremist_sym_upper}, let $W\in\Gra$ satisfy $\Lone{W} > 1 - \theta(k)$. Since $\Lone{1-W} = 1 - \Lone{W} <\theta(k)$, we have $\dest_{\rul}(1-W) = 0$ and therefore~\ref{item:extremist_sym} implies $\dest_{\rul}(W) = 1$.

  Part~\ref{item:extremist_sym_half} follows from the fact that the constant-$1/2$ graphon $W$ satisfies $\dest_{\rul}(W) = 1/2$.
\end{proof}

We prove the following lower bound on the density threshold. In particular, we show that the density threshold is strictly positive for $k \ge 5$.

\begin{thm} \label{thm:extremist-threshold-bound}
Let $k\ge5$ be an integer and define $\sigma(k) := \max\{2^{-4-24/(k-5)}, 1/2-\sqrt{48/k}\}$ for $k\ge 6$ and $\sigma(5)=1/30$. For the extremist rule $\rul$ of order $k$ we have that
\[\theta(k)\ge\sigma(k)>0.\]
\end{thm}

\begin{proof}
  
The following is our key claim.

\begin{claim} \label{claim:extremist-threshold-vel-int}
Given a graphon $U \in \Gra$ satisfying $\Lone{U} \le \sigma(k)$, we have
\begin{equation} \label{eq:extremist-threshold-vel-int}
\int \vel[\rul] U \D\pi^2 \le -\binom{k}{2}\Lone{U}.
\end{equation}
\end{claim}

Before we prove this claim, let us first demonstrate how it implies our theorem. Fix $W \in \Gra$ satisfying $\Lone{W} \le \sigma(k)$. Consider the function $f(t) := \Lone{\traj{t}W}$ defined for $t \in [0, \infty)$. Aiming to apply Lemma~\ref{lem:FODEonesidedBound} with $D = \sigma(k)$, we note that $f(0) \le \sigma(k)$ and $f(t) \ge 0$ for all $t \ge 0$. We argue that $f'(t) = \int \vel[\rul] \traj{t}W \D \pi^2$. Indeed, we have that
\[
  f'(t)=\lim_{\varepsilon\rightarrow 0}\frac{\Lone{\traj{t+\varepsilon}W}-\Lone{\traj{t}W}}{\varepsilon}=\lim_{\varepsilon\rightarrow0}\int\frac{\traj{t+\varepsilon}W-\traj{t}W}{\varepsilon}\D \pi^2.
\]
Now by Lemma~\ref{lem:traj-cont-Linfty} the integrand is $O(\varepsilon)$-close to $\vel[\rul]\traj{t}W$ in the $L^{\infty}$-norm (and thus in the $L^1$-norm), so in the limit we get $f'(t) = \int \vel[\rul] \traj{t}W \D \pi^2$. Together with \eqref{eq:extremist-threshold-vel-int} we get that $f'(t) \le -\bik f(t)$ whenever $f(t) \le \sigma(k)$, so by Lemma~\ref{lem:FODEonesidedBound} we have $f(t) \le f(0)e^{-\binom{k}{2}t}$ for all $t \in [0,\infty)$. Therefore, it follows that $\dest_\rul(W) = 0$ and we conclude that $\theta(k) \ge \sigma(k) > 0$.

It remains to prove Claim~\ref{claim:extremist-threshold-vel-int}.

\begin{claimproof}[Proof of Claim~\ref{claim:extremist-threshold-vel-int}]
Set $p := \Lone{U}$. Sample $\ff \sim \G(k,U)$ and sample $\hh$ according to the distribution $\Pc{\hh = H}{\ff = F} = \rul_{F,H}$. We first obtain an upper bound for $\E e(\hh)$. Since $\Ec{e(\hh)}{\ff=F} = 0$ for any $F\in\lgr{k}$ with fewer than $\frac{1}{2}\binom{k}{2}$ edges, we have
\begin{equation} \label{eq:expectation_H-extremist-upper}
\E e(\hh) \le \binom{k}{2}\Prob \left(e(\ff) \ge \frac{1}{2}\binom{k}{2}\right).
\end{equation}

Now suppose that we have
\begin{equation} \label{eq:prob_F-upper}
\Prob \left(e(\ff)\ge\frac{1}{2}\binom{k}{2}\right) \le \frac{p}{2}.
\end{equation}
Then by~\eqref{eq:vel_integral2} and~\eqref{eq:expectation_H-extremist-upper} we conclude that
\[
  \int \vel[\rul] U \D \pi^2 = 2\E e(\hh) - (k)_2 \cdot p \le p \left(\binom{k}{2} - (k)_2\right) = -p \binom{k}{2},
\]
which yields~\eqref{eq:extremist-threshold-vel-int}. Hence, it remains to prove~\eqref{eq:prob_F-upper}.

We begin with $k=5$ and $p \le 1/30$. Note that if $\ff$ has at least $\frac{1}{2}\bik = 5$ edges then it has a matching with $2$ edges.
There are $\binom{5}{2}\binom{3}{2}/2=15$ such matchings in $K_5$ and each of them appears in $\ff$ with probability $p^2$ (as each edge occurs independently with probability $p$), so we have
\[
  \Prob \left(e(\ff) \ge  \frac{1}{2}\binom{k}{2}\right) \le \prob{\ff \text{ contains a $2$-matching}} \le 15p^2 \le \frac{p}{2}.
\]

Next, we consider $k\ge6$ and $p \le 2^{-4-24/(k-5)}$. We partition $K_k$ into disjoint maximum matchings $M_1,\dots,M_r$, with $r=\binom{k}{2}/\flo{\frac{k}{2}} \in \{k-1,k\}$. If $e(\ff) \ge \binom{k}{2}/2$ then there exists $i\in [r]$ such that
\[
  e(\ff\cap M_i) \ge \cei{\frac{\bik}{2r}} = \cei{\frac{\flo{k/2}}{2}}= \flo{\frac{k+2}{4}} =: k'\;.
\]
Therefore, using $r \le k \le 2^{k/2}$ and $\binom{\lfloor k/2 \rfloor}{k'} \le 2^{k/2}$, we have
\[
  \Prob \left(e(\ff) \ge \frac{1}{2}\binom{k}{2} \right) \le \sum_{i\in[r]} \Prob \left(e(\ff\cap M_i) \ge k'\right) \le r \binom{\lfloor k/2 \rfloor}{k'}p^{k'} \le 2^kp^{(k-1)/4} \le \frac{p}{2}.
\]

Finally, we consider $2^{-4-24/(k-5)} \le p \le 1/2-\sqrt{48/k}$; here we may assume that $k\ge192$. We consider the function $f(G) := e(G)/(v(G)-1)$. Since $f$ varies by at most $1$ under the editing of edges incident to any fixed vertex, by~\cite[Theorem 10.3]{Lovasz2012} we obtain
\[
  \prob{e(\ff) \ge \frac{1}{2}\binom{k}{2}} = \prob{f(\ff) \ge \E f(\ff) + (\tfrac{1}{2}-p)k/2} \le \exp \left(-\frac{(1/2 -p)^2k}{8}\right).
\]
Now we have $e>2$ and $\frac{(1/2 -p)^2 k}{8} \ge 6 \ge 5+\frac{24}{k-5}$, so we obtain
\[
  \prob{e(\ff) \ge \bik/2} \le 2^{-5-24/(k-5)} \le \frac{p}{2}
\]
as required.
\end{claimproof}
This completes our proof.
\end{proof}

Now we show that the results of Theorem~\ref{thm:extremist-threshold-bound} cannot be extended to $k\in \{3,4\}$.

\begin{thm} \label{thm:extremist-zero-threshold-bound}
$\theta(3) = \theta(4) = 0$.
\end{thm}

\begin{proof}
Let $\rul$ be the extremist rule of order $k\in\{3,4\}$. Define
\begin{align*}
B_1 &:= \left\{(F,a,b) : F\in\lgr{k}, e(F)<\frac{1}{2}\binom{k}{2}, ab\in E(F)\right\}\ \textrm{ and } \\
B_2 &:= \left\{(F,a,b) : F\in\lgr{k}, e(F)>\frac{1}{2}\binom{k}{2}, ab\notin E(F)\right\}.
\end{align*}
By~\eqref{eq:velocity} and~\eqref{eq:extremist-rule}, for each $W\in\Gra$ and each $(x,y)\in\Omega^2$ we have
\begin{equation} \label{eq:extremist-velocity}
\vel[\rul]W(x,y) = \sum_{(F,a,b) \in B_2}\tindr{(x,y)}(F^{a,b},W) - \sum_{(F,a,b) \in B_1}\tindr{(x,y)}(F^{a,b},W).
\end{equation}

We shall construct a sequence $(W_n)_{n\in\N}$ of graphons so that $\Lone{W_n}\to0$ as $n\to\infty$ and for each $n\in\N$ we have $\cutn{\traj{t}W_n}\nrightarrow0$ as $t\to\infty$. Fix an arbitrary sequence $(\Omega_n)_{n\in\N}$ of subsets of $\Omega$ such that for each $n\in\N$ we have $\pi(\Omega_n)=1/n$ and $\Omega_{n+1}\subseteq\Omega_n$. For each $n\in\N$ set $S_n := ((\Omega\setminus\Omega_n)\times\Omega_n)\cup(\Omega_n\times(\Omega\setminus\Omega_n))$ and $\ca_n := \{ W\in\Gra : W{\restriction_{S_n}} = 1 \}$. For each $n\in\N$ define the graphon $W_n := \indic_{S_n}$. Note that $\Lone{W_n}=\frac{2}{n} \left(1-\frac{1}{n}\right)\to0$ as $n\to\infty$.

Fix $n\in\N$. We have $W_n\in\ca_n$. By Theorem~\ref{thm:flow} there is a unique trajectory $\traj{t}W_n$ on $[0,\infty)$. Let $U\in\ca_n$. For any $(F,a,b) \in B_1 \cup B_2$ and $(x,y)\in S_n$ we have $\tindr{(x,y)}(F^{a,b},U)=0$, so by~\eqref{eq:extremist-velocity} we have $(\vel[\rul]U){\restriction_{S_n}} = 0$. Now by applying Proposition~\ref{prop:zeroonasection} with $\ca_n$ we have $\traj{t}{W_n}\in\ca_n$ for all $t \in [0,\infty)$. We have $\cutn{\traj{t}W_n}\ge\frac{2}{n} \left(1-\frac{1}{n}\right)$ for all $t \in [0,\infty)$; in particular, $\cutn{\traj{t}W_n}\nrightarrow0$ as $t\to\infty$. Hence, we have $\theta(3) = \theta(4) = 0$.
\end{proof}

We also prove an upper bound on the density threshold.

\begin{thm} \label{thm:extremist-threshold-upper}
There exists $K \in \N$ such that for every $k \ge K$, setting $C_k := (k)_2^22^{\binom{k}{2}-1}$, we have $\theta(k) \le \frac{1}{2}-\frac{(k)_2}{2\cdot10^5C_k}$.
\end{thm}

\begin{proof}
Let $C$ be the constant returned by Theorem~\ref{thm:berry-esseen} and set $K := 400C^2$. Let $\rul$ be the extremist rule of order $k \ge K$ and set $f(k) := \frac{(k)_2}{2\cdot10^5C_k}$. Suppose for a contradiction that $\theta(k) > 1/2 - f(k)$.

Set $\delta := \frac{1}{2}\sqrt{1-2\theta(k)+2f(k)}$ and $\gamma := 1/2-\delta$. Fix a subset $\Omega'\subseteq\Omega$ such that $\pi(\Omega')=\gamma$. Set $S:=((\Omega\setminus\Omega')\times\Omega')\cup(\Omega'\times(\Omega\setminus\Omega'))$. Let $W := \indic_S$. Note that $\Lone{W} = \theta(k) - f(k) < \theta(k)$. Sample $\ff \sim \G(k,W)$ and sample $\hh$ according to the distribution $\Pc{\hh = H}{\ff = F} = \rul_{F, H}$. Note that
\begin{equation} \label{eq:expectation_H-extremist-lower}
\E e(\hh) \ge \binom{k}{2} \prob{e(\ff) > \frac{1}{2}\binom{k}{2}}.
\end{equation}
Set $Y := e(\ff)$. By the structure of $W$ we have $B \sim \Bin(k,\gamma)$ such that $Y = B(k-B)$. Write $B = \sum_{i=1}^k B_i$ as the sum of $k$ independent random variables $B_i\sim\ber(\gamma)$. Set $X_i := B_i-\gamma$ for $i\in[k]$ and $X := \sum_{i=1}^k X_i$. We have
\begin{align*}
\prob{Y > \frac{1}{2}\binom{k}{2}} = \prob{ \left|B - \frac{k}{2}\right| < \frac{\sqrt{k}}{2}} &\ge \prob{|X| < \frac{\sqrt{k}}{2} - k\delta} \\
&\ge \prob{\frac{|X|}{\sqrt{\var(X)}} < \frac{1 - \sqrt{4kf(k)}}{\sqrt{1-4f(k)}}} \\
&\ge \prob{\frac{|X|}{\sqrt{\var(X)}} < 0.99}.
\end{align*}
Now by Theorem~\ref{thm:berry-esseen} we have
\[
  \prob{\frac{|X|}{\sqrt{\var(X)}} < 0.99} \ge \prob{Z < 0.99} - \frac{1}{9} \ge 0.55,
\]
where $Z\sim N(0,1)$ is a standard Gaussian random variable. Then by~\eqref{eq:vel_integral2} and~\eqref{eq:expectation_H-extremist-lower} we have
\begin{equation} \label{eq:extremist-postive-velocity-bipartite-near-half}
\int \vel W \D\pi^2 \ge (k)_2 \left(\prob{Y > \frac{1}{2}\binom{k}{2}} - \Lone{W}\right) \ge \frac{(k)_2}{100}.
\end{equation}

Set $T:=\frac{1}{200C_k}$. By~\eqref{eq:LinfPhi_eps} in Lemma~\ref{lem:traj-cont-Linfty} we have
\begin{equation*}
\int \left( \traj{T}{W} - W - T\vel[\rul]W \right) \D\pi^2 \le \Linf{\traj{T}{W} - W - T\vel[\rul]W} \le C_k(k)_2T^2,
\end{equation*}
so by~\eqref{eq:extremist-postive-velocity-bipartite-near-half} we have
\begin{align*}
\Lone{\traj{T}{W}} = \int \traj{T}{W} \D\pi^2 &\ge \int W \D\pi^2 + T \int \vel W \D\pi^2 - C_k(k)_2T^2 \\
&\ge \Lone{W} + \frac{(k)_2}{2\cdot10^4C_k} - \frac{(k)_2}{4\cdot10^4C_k} \ge \frac{1}{2} + \frac{(k)_2}{8\cdot10^4C_k} > 1-\theta(k).
\end{align*}
Hence, by Proposition~\ref{prop:extremist_symm}\ref{item:extremist_sym_upper} we have $\dest(W) = \dest(\traj{T}W) = 1$. On the other hand, $\Lone{W} < \theta(k)$ implies that $\dest(W) = 0$. This is a contradiction, so we must have $\theta(k) \le 1/2 - f(k)$.
\end{proof}

\subsection{Trajectory rescaling for large \texorpdfstring{$k$}{k}}\label{ssec:rescalingextremist}
Observe that the order of a rule can significantly restrict the maximum rate of change in its trajectories. Indeed, suppose that $\rul$ is a rule of order $k\in\N$. Then, by~\eqref{eq:vel_bound_by_W} we have $-(k)_2 \le \vel[\rul] W \le (k)_2$ for all graphons $W$. To adjust for the fact that larger values of $k$ permit greater rates of change, we shall introduce the following limit notion of rescaling. Let $(\rul_k)_{k\ge2}$ be a sequence of rules such that $\rul_k$ is a rule of order $k$ for each $k\ge2$. We say that $(\rul_k)_{k\ge2}$ is \emph{rescalable} at a graphon $W$ if the Banach-space limit
\[
  \lim_{k\to\infty} \trajrul{\rul_k}{t/(k)_2}{W}
\]
exists for all $t\in[0,\infty)$. Then, we write $\Psi^t(W):=\lim_{k\to\infty} \trajrul{\rul_k}{t/(k)_2}{W}$ and call this a \emph{rescaled trajectory}. Let $\mathcal{B}\subseteq\Gra$ be the set of graphons at which $(\rul_k)_{k\ge2}$ is rescalable. We define the \emph{rescaling core} of $(\rul_k)_{k\ge2}$, which we denote by $\mathcal{B}^\bullet$, to be the set of elements in the interior of $\mathcal{B}$ (with respect to the topology induced by the cut norm) at which the maps $W \mapsto \Psi^t(W)$, for $t\ge 0$, are continuous (with respect to the topology induced by the cut norm). In particular, the rescaling core of $(\rul_k)_{k\ge2}$ is a set of graphons $W$ such that $\Psi^{t(k)_2}(W)$ is a good approximation for $\trajrul{\rul_k}{t}{(\tilde{W})}$ when $k$ is sufficiently large and $\tilde{W}$ is sufficiently close to $W$. (Let us note that in principle, we might investigate rescaled trajectories outside of the rescaling core as well, and this indeed seems to be doable using the main result from~\cite{MR4313198}. However, the favourable continuity property is lost in this case, and in particular such rescaled trajectories would have no connection to flip processes on finite graphs, as provided by Theorem~\ref{thm:correspondenceInformal}.)

Theorem~\ref{thm:extremistrescaled} characterizes the rescaling core of the sequence of extremist rules and explicitly describes the rescaled trajectories.

\begin{thm}\label{thm:extremistrescaled}
	The rescaling core of the sequence $(\rul_k)_{k\ge2}$ of extremist rules is the set of graphons whose density is not $\frac{1}{2}$. The following holds for each graphon $W\in\Gra$.
	\begin{romenumerate}
		\item\label{en:rescaledBIGGER} If $\Lone{W} > \frac{1}{2}$ then $\Psi^tW = 1 - e^{-t} \cdot (1-W)$ for each $t \ge 0$.
		\item\label{en:rescaledSMALLER} If $\Lone{W} < \frac{1}{2}$ then $\Psi^tW = e^{-t} \cdot W$ for each $t \ge 0$.
	\end{romenumerate}
\end{thm}
We only prove Theorem~\ref{thm:extremistrescaled}\ref{en:rescaledBIGGER}. Indeed, then Theorem~\ref{thm:extremistrescaled}\ref{en:rescaledSMALLER} follows by symmetry. Further, the set of graphons of density other than $\frac12$ is open and the maps $W \mapsto \Psi^t(W)$ are continuous. So, it only needs to be argued that the rescaling core does not contain any graphon of density $\frac12$. This is because such a graphon can be approximated by graphons of density more than $\frac12$ and less than $\frac12$ at the same time, which are mapped by $\Psi^{3}(\cdot)$ to a graphon of density more than $0.9$ and less than $0.1$, respectively, demonstrating a discontinuity.
\begin{proof}[Proof of Theorem~\ref{thm:extremistrescaled}\ref{en:rescaledBIGGER}]

Fix $\eps>0$ and consider an extremist rule $\rul_k$ of sufficiently large order $k$. We write $\vel[k]$ and $\trajrul{k}{}$ for the velocity operator and the trajectory of $\rul_k$, respectively. Let $\cu$ denote the set of all graphons with density at least $\frac12+\eps$. We begin with the following claim, which approximates $\vel[k]$ for graphons in $\cu$.

\begin{claim} \label{claim:extremist-rescaled-vel-approx}
If $U\in\cu$ then we have
\[
  \Linf{\vel[k]U - (k)_2\cdot(1-U)}=o_k(k^2)\;.
\]
\end{claim}

\begin{claimproof}
Set $d':=\Lone{U}$. Observe that we have
\begin{equation}\label{eq:klssH}
\sum_{F\in \lgr{k}:ab\notin E(F)} (\Troot{F}{a,b}W)(x,y)=1-W(x,y)\;.
\end{equation}
For $1 \le a\neq b \le k$ and a labelled graph $J$ on $[k]\setminus\{a,b\}$, let $\mathcal{F}^{a,b}_J\subset \lgr{k}$ be the set of all graphs on $[k]$ from which the removal of $a$ and $b$ yields $J$. Observe that we have
\begin{equation}\label{eq:sum-derooted}
\tind(J,U) = \sum_{F \in \mathcal{F}^{a,b}_J} \tind(F,U)\;.
\end{equation}

For $\alpha,\beta\in[0,1]$ and $\ell\in \N$, let $\lgr{\ell}^{\alpha,\beta}\subset \lgr{\ell}$ be the set of all graphs whose density is $\alpha\pm \beta$. Since $\tind(J,U)$ represents the probability that $\G(k,U)=J$, Lemma~\ref{lem:sampling} tells us that
\begin{equation}\label{eq:Itellyou}
\sum_{J\in\lgr{k-2}\setminus \lgr{k-2}^{d',\eps/3}} \tind(J,U) =o_k(1/k^2)
\end{equation}
Since any graph $F\in\lgr{k}^{d',\eps/2}$ has density greater than $\frac12$, we have $\rul_{F,K_k}=1$. On the other hand, since $k$ is sufficiently large and at most $2k=o_k(k^2)$ edges are incident to $a$ or $b$, a graph $F\in\mathcal{F}^{a,b}_J$ does not have density $d'\pm \eps/2$ only if $J$ does not have density $d'\pm \eps/3$. Hence, we may rewrite~\eqref{eq:velocity} as
\begin{align*}
\vel[k]U(x,y) &= \sum_{1 \le a\neq b \le k} \sum_{F\in\lgr{k}:ab\notin E(F)} (\Troot{F}{a,b}U)(x,y)\pm 2\sum_{1 \le a\neq b \le k} \sum_{F\in\lgr{k}\setminus\lgr{k}^{d',\eps/2}}   (\Troot{F}{a,b}U)(x,y) \\
\justify{by~\eqref{eq:klssH}}&=(k)_2\cdot (1-U(x,y))\pm 2\sum_{1 \le a\neq b \le k} \sum_{J\in\lgr{k-2}\setminus \lgr{k-2}^{d',\eps/3}} \sum_{F\in\mathcal{F}^{a,b}_J}\tind(F,U)\\
\justify{by~\eqref{eq:sum-derooted} and~\eqref{eq:Itellyou}}&=(k)_2\cdot (1-U(x,y)) \pm o_k(1))\;,
\end{align*}
finishing the proof.
\end{claimproof}

Now we apply Claim~\ref{claim:extremist-rescaled-vel-approx} to show that trajectories which start in $\cu$ are confined to $\cu$.

\begin{claim}
If $U\in\cu$, then we have $\trajrul{k}{t}{U}\in\cu$ for all $t \ge 0$.
\end{claim}

\begin{claimproof}
Let $U\in\cu$ and suppose that $\trajrul{k}{t}{U}\notin\cu$ for some $t \ge 0$. Set $T:=\inf\{t \ge 0 : \trajrul{k}{t}{U}\notin\cu\}$. By the continuity of $\trajrul{k}{\cdot}{U}$, the graphon $\trajrul{k}{T}{U}$ has density $\frac12+\eps$. Set $S= \left( (k)_2^3 2^{\binom{k}{2}} \right)^{-1}$ and let $0 < s \le S$. By~\eqref{eq:LinfPhi_eps} for $s$ and Claim~\ref{claim:extremist-rescaled-vel-approx}, we have
\[
  \Linf{\frac{\trajrul{k}{T+s}{U} - \trajrul{k}{T}{U}}{s} - (k)_2 \cdot (1-\trajrul{k}{T}{U})} = o_k(k^2).
\]
Then, we obtain
\[
  \Lone{\trajrul{k}{T+s}{U}} = \Lone{\trajrul{k}{T}{U}} + s((k)_2\Lone{1-\trajrul{k}{T}{U}}-o_k(k^2)) \ge \Lone{\trajrul{k}{T}{U}} = \frac{1}{2} + \eps,
\]
so we have $\trajrul{k}{T+s}{U}\in\cu$. This contradicts the definition of $T$, so $\trajrul{k}{t}{U}\in\cu$ for all $t \ge 0$.
\end{claimproof}

Let $W\in\cu$. We first consider the integral equation $\Gamma_k^\tau:=W+\int_{s=0}^\tau \mathfrak{Z}_k(\Gamma^s)\D s$, where $\mathfrak{Z}_k(U)=(k)_2\cdot(1-U)$. The solution of this integral equation is
\[
  \Gamma_k^\tau = 1 - e^{-(k)_2 \tau} \cdot (1-W)\;.
\]
Recall that in fact $\trajrul{k}{\cdot}{W}$ is the solution of the integral equation~\eqref{eq:integralequation}, that is,
\[
  \trajrul{k}{\tau}{W} = W + \int_0^\tau \vel[k] \left(\trajrul{k}{s}{W}\right) \D s\;.
\]
Now $\mathfrak{Z}_k(\cdot)$ is $(k)_2$-Lipschitz, both $\Gamma_k^\cdot$ and $\trajrul{k}{\cdot}{W}$ are confined to $\cu$ and Claim~\ref{claim:extremist-rescaled-vel-approx} tells us that $\Linf{\vel[k](\cdot)-\mathfrak{Z}_k(\cdot)} = o_k(k^2)$ on $\cu$, so Lemma~\ref{lem:perturb} gives us
\begin{equation*}
\Linf{\trajrul{k}{\tau}{W} - \Gamma^\tau_k} \le o_k(k^2)\tau \cdot e^{(k)_2 \tau}\;.
\end{equation*}
Hence, we conclude that for any fixed $t\ge 0$, we have that $\trajrul{k}{t/(k)_2} W$ converges uniformly to $1 - e^{-t}\cdot (1-W)$ as $k\to\infty$.
\end{proof}

\subsection{Stabilization in an extremist flip process}\label{ssec:extremistStabilization}
Let us consider a flip process of order $k$ with rule $\rul$. Let $G$ be an $n$-vertex graph ($n \ge k$), and $G_0:=G,G_1,G_2,\ldots$ be an evolution of the graph $G$ with respect to the flip process with rule $\rul$. The \emph{stabilization time} $\kappa_{\rul}(G)$ is a random variable defined as $\kappa_{\rul}(G) = \min\{\ell:G_{\ell}=G_{\ell+1}=G_{\ell+2}=\ldots\}$ if it exists and $\kappa_{\rul}(G) = \infty$ otherwise. For example, if $\rul_{F,F}=0$ for all $F\in\lgr{k}$ then each step of the flip process changes the graph and $\kappa_{\rul}(G)=\infty$ almost surely.

We now consider the concept of stabilization time in the context of an extremist flip process of a given order $k\ge3$. Suppose that the initial graph has $n$ vertices. It can be shown that the stabilization time is finite almost surely, and further, if $n\ge n_0(k)$ is sufficiently large, stabilization occurs (almost surely) at the moment the flip process reaches the edgeless or the complete graph (which are graphs at which stabilization obviously occurs).

We leave it as an interesting question to find universal bounds on stabilization times.
\begin{qu}\label{qu:stabilizationextremist}
For each $k\ge 3$ with the corresponding extremist rule $\rul_k$, we seek a function $b_k:\N\rightarrow\N$ growing as slowly as possible so that \[
  \min_{G \in \lgr{n}}\prob{ \kappa_{\rul_k}(G)\le b_k(n) } \to 1, \quad \text{ as } n \to \infty \;.
\]
\end{qu}

It seems very likely that quasirandom graphs of edge density around $\frac12$ have the largest stabilization time, and that we have to take $b_k(n)\approx\exp(n^C)$ for some $C\in (0,2]$. In particular, the number of steps goes far beyond $\Theta(n^2)$, and so this question is outside the framework of our graphon trajectories.\footnote{Recall Theorem~\ref{thm:correspondenceInformal}.}

\section{Monotone flip processes}\label{sec:monotone}
Broadly speaking, monotone flip processes are those where the number of edges does not decrease in any step. Formally, we define a few different notions of monotonicity.
\begin{defi}\label{def:monotone-rule}
Let $\rul$ be a rule of order $k\in\N$. We say that $\rul$ is
\begin{itemize}
	\item \emph{non-decreasing} if for each $F\in\lgr{k}$ we have $\rul_{F,H} = 0$ whenever $H$ is not a (labelled) supergraph of $F$.
	\item \emph{increasing} if it is non-decreasing and in addition for each $F\in \lgr{k}\setminus\{K_k\}$ there exists a proper supergraph $H$ of $F$ with $\rul_{F,H} > 0$.
	\item \emph{density-non-decreasing} if for each $F\in\lgr{k}$ we have $\sum_{H\in\lgr{k}}\rul_{F,H}e(H) \ge e(F)$.
	\item \emph{density-increasing} if it is density-non-decreasing and further for each $F\in\lgr{k}\setminus\{K_k\}$ the above inequality is strict.
	\end{itemize}
\end{defi}
In particular, a non-decreasing rule is density-non-decreasing and an increasing rule is density-increasing.

The following proposition asserts that non-decreasing and density-increasing flip processes are convergent, but this need not be true for density-non-decreasing flip processes.
\begin{prop}\label{prop:monotoneflipprocess}
$~$
\begin{romenumerate}
	\item\label{en:monotone1} Each trajectory of a non-decreasing rule $\rul$ has a destination.
	\item\label{en:monotone2} Let $\rul$ be a density-increasing rule. The destination of each trajectory is the constant-$1$ graphon. The age of each graphon except the constant-$1$ graphon is finite.
	\item\label{en:monotone3} There is a density-non-decreasing rule with a non-convergent trajectory.
\end{romenumerate}
\end{prop}
\begin{proof}[Proof of Proposition~\ref{prop:monotoneflipprocess}\ref{en:monotone1}]
Let $W\in\Gra$ be a graphon. Observe that since $\rul$ is non-decreasing, the negative terms of~\eqref{eq:velocity} have coefficient zero. Hence, the velocity is pointwise non-negative. In other words, for each point $(x,y)$, the function $t\mapsto\traj{t}{W}(x,y)$ is non-decreasing. In this case it is obvious that the destination is the pointwise limit of $\traj{t}{W}$.
\end{proof}

\begin{proof}[Proof of Proposition~\ref{prop:monotoneflipprocess}\ref{en:monotone2}]
Suppose $\rul$ is of order $k$. Let
\[
  \delta = \min_{F \in \lgr{k}\setminus \{K_k\}} \sum_{H \in \lgr{k}} \rul_{F,H}(e(H)-e(F)).
\]
Since $\rul$ is density-increasing, we have $\delta>0$. We shall show that for every graphon $W$ we have
\begin{equation}\label{eq:change_toward_one}
\int \vel W \D\pi^2 \ge 2\delta (1-t(K_k,W))\;.
\end{equation}
Let $(\ff,\hh)$ be a pair of random graphs in $\lgr{k}$ such that $\ff \sim \G(k,W)$ and $\Pc{\hh = H}{\ff = F} = \rul_{F,H}$ for all $F,H\in\lgr{k}$. For $F \neq K_k$ we have
\begin{equation}\label{eq:yImrl}
  \Ec{e(\hh)}{\ff=F} = \sum_{H \in \lgr{k}} \rul_{F,H} e(H) \ge e(F) + \delta.
\end{equation}

Now by~\eqref{eq:vel_integral} we have
\begin{align*}	
\int \vel  W \D\pi^2
 &= 2\E \left[ e(\hh) - e(\ff) \right] = 2\sum_{F \in \lgr{k}} t(F,W) \Ec{e(\hh) - e(\ff)}{\ff=F} \\
 &= 2\sum_{F \in \lgr{k}} t(F,W) \left[ \Ec{e(\hh)}{\ff=F} - e(F) \right] \\
 &\geByRef{eq:yImrl} 2\sum_{F \neq K_k} t(F, W) \delta = 2 \left( 1 - t(K_k, W) \right)\delta\;,
\end{align*}
which shows~\eqref{eq:change_toward_one}.
	
The growth given by~\eqref{eq:change_toward_one} allows us to prove the statement. Indeed, let $U$ be an arbitrary graphon. Let $f:[-\age{U},+\infty)\rightarrow [0,1]$ be defined by $f(t):=\int \traj{t}{U}\D\pi^2$. Then~\eqref{eq:change_toward_one} gives
	\begin{equation*}
	  \frac{\D}{\D t}f(t)\ge 2\delta(1-t(K_k, \traj{t}{U})) \geBy{Prop~\ref{prop:KruskalKatona}}
	2\delta(1-f(t)^{k/2})\;.
	\end{equation*}
	Firstly, this implies that $f(t)\rightarrow 1$ as $t\rightarrow \infty$, i.e., that the destination of $U$ is the constant-$1$ graphon. Secondly, it gives that $f(t)$ is increasing, and hence for $t<0$, we have that $\frac{\D}{\D t}f(t)\ge 2\delta(1-f(0)^{2/k})$. That means that unless $f(0)=1$, the derivative of $f$ in the negative domain is bounded from below by a positive constant. Since $f\in[0,1]$, this means that the age is finite.
\end{proof}

\begin{proof}[Sketch of proof of Proposition~\ref{prop:monotoneflipprocess}\ref{en:monotone3}]
	We build on Section~\ref{FIRSTPAPER.ssec:oscillatory} of~\cite{Flip1} where a flip process with a periodic trajectory was constructed. We briefly recall this construction. Fix a partition $\Omega = \Omega_1 \cup \Omega_2$ with $\pi(\Omega_1) = \pi(\Omega_2) = 1/2$. We take $k$ large and $\alpha>0$ small (probably $k=2000$ and $\alpha=0.01$ would work). Fix $r=0.1$, $a:=0.2$, $b:=0.8$. Let $\mathsf{circle}(a,b;r)$ be the circle centered at the point $(a,b)$ and of radius $r$. For input $F\in\lgr{k}$ the rule is as follows. If $F$ consists of two components $C_1$ and $C_2$, each of order $(\frac12\pm\alpha)k$, say $C_1$ of density $d_1$ and $C_2$ of density $d_2$ with the condition that the point $(d_1,d_2)$ is within distance of $0.01$ from $\mathsf{circle}(a,b;r)$ then we add or delete edges into $C_1$ and $C_2$ so that the expected change of edges in $C_1$ and $C_2$ is given by the first and the second coordinate of $(g+h)(d_1,d_2)$, respectively. Here, $g,h:\R^2\setminus\{(a,b)\}\rightarrow\R^2$  are defined as follows,
	\begin{align*}
	g(x,y)&=\frac{1}{\sqrt{(-y+b)^2+(x-a)^2}}\cdot (-y+b,x-a)\;,\\
	h(x,y)&= \left(\tfrac{r}{\sqrt{(x-a)^2+(y-b)^2}}-1\right)\cdot (x-a,y-b)\;.
	\end{align*}
	No edges between $C_1$ and $C_2$ are added. For other $F$s, the rule is arbitrary. Observe that values of $g$ are tangent vectors of circles centered at $(a,b)$. Furthermore, $h$ vanishes on $\mathsf{circle}(a,b;r)$ and pushes towards $\mathsf{circle}(a,b;r)$ elsewhere.
	
	Fix $(x,y)\in\mathsf{circle}(a,b;r)$.
	Consider a graphon $U$ taking value $x$ on $\Omega_1\times\Omega_1$, value $y$ on $\Omega_2\times\Omega_2$ and~$0$ elsewhere. If it were true that for each such graphon we would have with probability~1 that the sampled $F$ has the two components as above and $d_1=x$ and $d_2=y$ then $\mathsf{circle}(a,b;r)$ would indeed be a periodic trajectory of the flip process. However, because the law of large numbers is not perfect, the above statement holds only with probability~$1-o_k(1)$ and only with $d_1=x\pm o_k(1)$ and $d_2=y\pm o_k(1)$. Considering a trajectory that starts at a graphon corresponding to the point $(x,y)$, the correcting function $h$ pushes the trajectory towards $\mathsf{circle}(a,b;r)$ in case of any substantial deviations (this part of the argument needs the Poincar\'e-Bendixson Theorem and we refer to~\cite{Flip1}).
	
	The above flip process is clearly not density-non-decreasing. However, we can modify it as follows. Partition $\Omega=\Omega_1\cup\Omega_2\cup \Omega_\mathrm{reservoir}$ with $\pi(\Omega_1) = \pi(\Omega_2) = 1/4$ and $\pi(\Omega_\mathrm{reservoir})=1/2$. Now, the rule is as follows. If the sampled graph $F$ is such that it consists of three components, say $C_1$, $C_2$ and $C_\mathrm{reservoir}$, such that $C_1$ and $C_2$ have orders $(\frac14\pm\alpha)k$ and $C_\mathrm{reservoir}$ has order $(\frac12\pm\alpha)k$ and the densities of $C_1$ and $C_2$ are as above\footnote{Note that in this setting, the components $C_1$, $C_2$ and $C_\mathrm{reservoir}$ are uniquely determined.} and the density of $C_\mathrm{reservoir}$ is between $0.2$ and $0.8$ then we add or delete edges from $C_1$ and $C_2$ according to $(g+h)(d_1,d_2)$ and then add or delete edges from $C_\mathrm{reservoir}$ so that the total number of edges in this step does not change. If $F$ is not of this form then we have an idle step.
	
    Since the number of edges does not change, the flip process is non-density-decreasing. Starting with a graphon $W$ whose values on $\Omega_i\times \Omega_j$ ($i,j\in[2])$ are defined as for $U$, and further which is constant-$\frac{1}{2}$ on $\Omega_\mathrm{reservoir}^2$ and zero on $(\Omega_1\cup\Omega_2)\times \Omega_\mathrm{reservoir}^2$ we see that the densities on $\Omega_1^2$ and on $\Omega_2$ evolve close to $\mathsf{circle}(a,b;r)$ while $\Omega_\mathrm{reservoir}$ just serves to keep the overall density constant. In particular, this trajectory does not have a destination.
\end{proof}

Our next result says that for a non-decreasing rule the limit of the destinations of a convergent sequence of graphons is greater than or equal to the destination of the limit.

\begin{prop}\label{prop:NondecreasingSinksLower}
Let $\rul$ be a non-decreasing rule of order $k\in\N$ and $(U_n)_{n\in\N}$ be a sequence of graphons which converges to a graphon $U$ in the cut norm distance. Suppose that the sequence $(\dest(U_n))_{n\in\N}$ converges in the cut norm distance to a graphon $S$. Then $S\ge\dest(U)$.
\end{prop}

\begin{remark}\label{rem:converseMonotoneSink}
Observe that the converse inequality does not hold in general. Indeed, for each $k\ge 3$ consider a flip process of order $k$ in which a drawn graph containing a single edge and $k-2$ isolated vertices is replaced by $K_k$ while the step is idle for other drawn graphs. Let $(U_n)_{n\in\N}$ be the sequence of graphons where $U_n$ is the constant-$1/n$ graphon. On the one hand, this sequence of graphons converges to the constant-$0$ graphon $U$; here we have $\dest(U)=0$. On the other hand, $(\dest(U_n))_{n\in\N}$ is a sequence of constant-$1$ graphons, which clearly converges to the constant-$1$ graphon.

Nevertheless, we show in Theorem~\ref{thm:removalprocSinkLimit} that we do have equality for removal flip processes (which are also monotone, though non-increasing).
\end{remark}

\begin{proof}[Proof of Proposition~\ref{prop:NondecreasingSinksLower}]	
  We shall prove that for arbitrary sets $A,B\subset \Omega$ and an arbitrary $\eps>0$ we have $\int_{A\times B}\dest(U) \le \int_{A\times B}S +\eps$ (how this implies the proposition, is an exercise in measure theory which we leave to the reader). To this end, let $T$ be such that $\cutnd(\dest(U),\traj{T}{U})<\eps/2$. Define $\delta:=\eps/(4\exp(C_k T))$, where $C_k$ is from Theorem~\ref{thm:LipschTime}. Then by Theorem~\ref{thm:LipschTime}, for every $U_i$ for which $\cutnd(U_i,U)<\delta$, we have $\cutnd(\traj{T}U_i,\traj{T}U)<\eps/2$. In particular, this means that $\int_{A\times B}\dest(U) \le \int_{A\times B}\traj{T}U_i + \eps$. Since we have a non-decreasing flip process, $\traj{T}U_i \le \dest(U_i)$ and hence $\int_{A\times B}\dest(U) \le \int_{A\times B}\dest(U_i) + \eps$. As $\dest(U_1),\dest(U_2),\ldots\tocutn S$, we get $\int_{A\times B}\dest(U) \le \int_{A\times B}S + \eps$ as needed.
\end{proof}

\section{Component completion flip processes} \label{sec:componentcompletion}
In this section we define the component completion rule and prove Theorem~\ref{thm:component-completion-dests}, which characterizes the destinations of the component completion rule. The component completion rule completes the components of the drawn graph $F$ by adding edges between all non-edge pairs of vertices in the same connected component of $F$. We give a formal definition below.

\begin{defi}[Component completion graph]
Let $k\in\N$ and $F\in\lgr{k}$. For distinct $a,b\in[k]$ write $a \sim_F b$ to mean that there is a path between $a$ and $b$ in $F$. The \emph{component completion graph} $\cc(F)$ of $F$ is the labelled graph on vertex set $[k]$ with edge set $\{ab:a,b\in[k],a\neq b,a \sim_F b\}$.
\end{defi}

\begin{defi}[Component completion rule]
Let $k\in\N$. A rule $\rul$ of order $k$ is \emph{component completion} if for all $F,H\in\lgr{k}$ we have
\begin{equation} \label{eq:component-completion-rule}
\rul_{F,H}=
\begin{cases}
1&\textrm{ if }H=\cc(F), \\
0&\textrm{ otherwise }.
\end{cases}
\end{equation}
\end{defi}

To prepare for the statement and proof of Theorem~\ref{thm:component-completion-dests}, we give definitions of connectedness, component decomposition and component completeness for graphons.

\begin{defi}[Connectedness] \label{def:graphon-connected}
A graphon $W$ is \emph{disconnected} if $W = 0$ or there is a subset $A \subseteq \Omega$ with $\pi(A) \in (0,1)$ such that $W = 0$ a.e.\ on $A \times \left( \Omega \setminus A \right)$. Otherwise, $W$ is \emph{connected}.
\end{defi}

\begin{defi}[Component decomposition] \label{def:graphon-component-decomposition}
A \emph{component decomposition} of a graphon $W$ is a partition $\Omega = \bigsqcup_{i=0}^{K}\Omega_i$ into disjoint measurable sets $\Omega_i$ with $0 \le K \le \infty$ such that the following hold.
\begin{enumerate}[label=(GC\arabic{*})]
	\item \label{item:W-component-decomposition-measure} $\alpha_i := \pi(\Omega_i) > 0$ for all $i \ge 1$ but $\Omega_0$ may be empty.
	\item \label{item:W-component-decomposition-connected} For all $i \ge 1$ the graphon $W{\restriction_{\Omega_i^2}}$ on $ \left( \Omega_i , (\alpha_i^{-1} \pi){\restriction_{\Omega_i}} \right)$ is connected.
	\item \label{item:W-component-decomposition-no-edges-across}$W = 0$ a.e.\ on $\Omega^2 \setminus \left( \bigsqcup_{i=1}^{K}\Omega_i^2 \right)$.
\end{enumerate}
\end{defi}

\begin{defi}[Component completeness] \label{def:graphon-component-complete}
A graphon $W$ is \emph{component complete} if it has a component decomposition $\Omega = \bigsqcup_{i=0}^{K}\Omega_i$ such that $W{\restriction_{\Omega_i^2}} = 1$ for all $i \ge 1$.
\end{defi}

Theorem~\ref{thm:component-completion-dests} characterizes the destinations of the component completion rule.

\begin{thm} \label{thm:component-completion-dests}
Suppose that $\rul$ is a component completion rule of order $k \ge 3$. Then the following hold.
\begin{romenumerate}
	\item Every trajectory of $\rul$ has a destination.
	\item \label{item:component-completion-dests-complete} Every destination of $\rul$ is component complete. Moreover, the destination of a graphon with component decomposition $\Omega = \bigsqcup_{i=0}^{K}\Omega_i$ is component complete with the same component decomposition.
	\item Every component complete graphon is a destination of $\rul$.
\end{romenumerate}
\end{thm}

To prepare for the proof of Theorem~\ref{thm:component-completion-dests}, we state two lemmas. We begin with Lemma~5.5 from~\cite{Janson08Connectedness} about component decomposition for graphons.

\begin{lem} \label{lem:graphon-component-decomposition}
Every graphon has a component decomposition.
\end{lem}

Next, we give a technical lemma which, roughly speaking, says that every graphon with zero density of $3$-vertex induced paths is $\{0,1\}$-valued and in fact complete if connected. Its proof is given at the end of this section.

\begin{lem} \label{lem:no-cherries-zero-one}
Let $W$ be a graphon which satisfies
\begin{equation} \label{eq:cherry-density-zero}
\int_{\Omega^3} (1-W(x,y))W(x,z)W(y,z) \D\pi^3 = 0.
\end{equation}
Then $W(x,y) \in \{0,1\}$ a.e.\ on $\Omega^2$. If in addition $W$ is connected, then $W = 1$.
\end{lem}

Now we shall prove Theorem~\ref{thm:component-completion-dests}. Let us briefly describe the key ideas. The action of the non-decreasing rule $\rul$ is driven by the addition of edges between non-edge pairs of vertices which are connected by a path. Hence, the component decomposition of the initial graphon remains intact. Since our result is concerned with destinations of trajectories, which are a form of limit behaviour, $3$-vertex induced paths are particularly relevant because every induced path between non-adjacent vertices contains a $3$-vertex induced path. Lemma~\ref{lem:no-cherries-zero-one} implies that the components of a destination must be complete.

\begin{proof}[Proof of Theorem~\ref{thm:component-completion-dests}]
Since $\rul$ is a non-decreasing rule, every trajectory of $\rul$ has a destination by Proposition~\ref{prop:monotoneflipprocess}\ref{en:monotone1}. Let
\[
  B := \{(F,a,b) : F\in\lgr{k}, ab\notin E(F), a \sim_F b\}.
\]
By~\eqref{eq:velocity} and~\eqref{eq:component-completion-rule}, for each $W\in\Gra$ and each $(x,y)\in\Omega^2$ we have
\begin{equation} \label{eq:component-completion-velocity}
\vel[\rul]W(x,y)=\sum_{(F,a,b) \in B}\tindr{(x,y)}(F^{a,b},W).
\end{equation}
	
Let $U\in\Gra$, with a fixed component decomposition $\Omega = \bigsqcup_{i=0}^{K}\Omega_i$ of $U$ given by Lemma~\ref{lem:graphon-component-decomposition}. We set $\Lambda_i := ( \Omega_i \times ( \Omega \setminus \Omega_i ) ) \cup ( ( \Omega \setminus \Omega_i ) \times \Omega_i )$ for $i \in [K]$ and $\Lambda_0 := ( \Omega_0 \times \Omega ) \cup ( \Omega \times \Omega_0 )$. Set $\Lambda := \Omega^2 \setminus \left( \bigsqcup_{i=1}^{K}\Omega_i^2 \right) = \bigsqcup_{i=0}^{K}\Lambda_i$. For each $i \in [K]$ we show that the component completion rule does not change the value of $W$ at points $(x,y) \in \Lambda_i$, that is, pairs of vertices that go between $\Omega_i$ and its complement. This is intuitive because the rule would behave in the same way on graphs.

For $i \ge 0$ set $\ca_i := \{ W\in\Gra : W{\restriction_{\Lambda_i}}=0 \}$. We show that a trajectory starting at a graphon $W\in \ca_i$ stays in $\ca_i$ for every $i \ge 0$. Indeed for any $(F,a,b) \in B$ and $(x,y)\in\Lambda_i$ we have that $\tindr{(x,y)}(F^{a,b},W)=0$ exactly because any copy of $F^{a,b}$ sampled from $W$ with $x=a$ and $y=b$ sends the vertices $a$ and $b$ into different components. Hence, by~\eqref{eq:component-completion-velocity} we have $(\vel[\rul]W){\restriction_{\Lambda_i}}=0$. Then, by applying Proposition~\ref{prop:zeroonasection} with $\ca_i$ for each $i \in [K]\cup\{0\}$ we have $\traj{t}{U}=0$ on $\Lambda$ for all $t \in [0,\infty)$. Since $\traj{t}{U} \tocutn \dest(U) =: S$, it follows that $S=0$ a.e.\ on $\Lambda$. By~\eqref{eq:component-completion-velocity}, for all $t\in[0,\infty)$ we have $\vel[\rul]\traj{t}{U} \ge 0$ pointwise. Since $S = U + \int_{[0,\infty)}\vel[\rul]\traj{s}{U}{\D}s$, we have $S \ge U$. Hence, for all $i \ge 1$ the graphon $S{\restriction_{\Omega_i^2}}$ on $ \left(\Omega_i,(\alpha_i^{-1}\pi){\restriction_{\Omega_i}}\right)$ remains connected.

Define
\[
  Q := \left\{(x,y) \in \Omega^2 : (1-S(x,y)) \int_{z\in\Omega} S(x,z) S(y,z) > 0 \right\}.
\]
We aim to prove that $Q$ has measure zero. To this end, consider any $(x,y) \in Q$. Let $H$ be the labelled graph on $\{1,2,3\}$ with edge set $\{13,23\}$. Observe that we have $\tindr{(x,y)}(H^{1,2},S)>0$ and
\[
  \tindr{(x,y)}(H^{1,2},S)=\sum_{F}\tindr{(x,y)}(F^{1,2},S),
\]
where the sum is over all labelled graphs $F\in\lgr{k}$ which induce a labelled copy of $H$ on $\{1,2,3\}$. Hence, there is a labelled graph $F\in\lgr{k}$ which induces a labelled copy of $H$ on $\{1,2,3\}$ such that $\tindr{(x,y)}(F^{1,2},S)>0$. This implies that $\vel[\rul]S(x,y)>0$. Now since $S$ satisfies $\vel[\rul]S = 0$ by Proposition~\ref{prop:dests} and $\vel[\rul]S(x,y)>0$ for all $(x,y)\in Q$, the set $Q \subseteq \Omega^2$ has measure zero. Hence, we have~\eqref{eq:cherry-density-zero} for $S$ and $S{\restriction_{\Omega_i^2}}$ for all $i \ge 1$ and so by Lemma~\ref{lem:no-cherries-zero-one} we have $S = 1$ a.e.\ on $\Omega_i^2$ for all $1 \le i \le K$. Therefore, $S$ is a component complete graphon with component decomposition $\Omega = \bigsqcup_{i=0}^{K}\Omega_i$.

Finally, by~\ref{item:component-completion-dests-complete} a component complete graphon is clearly its own destination.
\end{proof}

To close this section, we provide the proof of Lemma~\ref{lem:no-cherries-zero-one}.

\begin{proof}[Proof of Lemma~\ref{lem:no-cherries-zero-one}]
Define the following sets.
\begin{align*}
M & := \{(x,y) \in \Omega^2 : W(x,y) \in (0,1)\}, \\
N_1 & := \left\{(x,y) \in \Omega^2 : \int_{z\in\Omega} W(x,z)W(y,z) \D\pi > 0\right\}, \\
N_2 & := \left\{(x,y) \in \Omega^2 : \int_{z\in\Omega} W(x,z)(1-W(y,z)) \D\pi > 0\right\}, \\
N & := N_1 \cup N_2, \\
P & := \{x \in \Omega : (x,y) \in M \setminus N \textrm{ for some } y \in \Omega\}.
\end{align*}
For each $(x,y)\notin N$ we have $\int_{z\in\Omega} W(x,z) \D\pi = 0$. By Fubini's theorem we have
\[
  0 \le \int_{M \setminus N} W(x,y) \D\pi^2 \le \int_{P} \int_{\Omega} W(x,y) \D\pi(y) \D\pi(x) = 0.
\]
Hence, we have $\pi^2(M\setminus N)=0$. By~\eqref{eq:cherry-density-zero} and Fubini's theorem we have $\pi^2(M\cap N_1)=\pi^2(M\cap N_2)=0$, so we have $\pi^2(M)=0$, i.e.\ $W(x,y) \in \{0,1\}$ a.e.\ on $\Omega^2$.
	
Now suppose that $W$ is connected. For each $x \in \Omega$ define the following.
\begin{align*}
S_x & := \{y \in \Omega : W(x,y) > 0 \}, \qquad S_x^{+} := S_x \cup \{x\}, \\
P_x & := \{ (y,z) \in S_x^{+} \times ( \Omega \setminus S_x^{+} ) : W(y,z) > 0 \}, \\
I(x) & := \int_{\Omega} W(x,y)\D\pi(y).
\end{align*}
In addition, we define the following.
\begin{align*}
S & := \{ (x,y) \in \Omega^2 : W(x,y) > 0 \}, \qquad J := \{x \in \Omega : I(x) > 0 \}, \\
P & := \{ (x,y,z) \in \Omega^3 : x \in J , (y,z) \in P_x \}, \\
Q & := \{ (x,y,z) \in \Omega^3 : x \in J , (y,z) \in S_x^2 \setminus S \}.
\end{align*}
Set $f(x,y,z) := (1-W(x,y)) W(x,z) W(y,z)$ for $(x,y,z) \in \Omega^3$.
	
Since $W$ is connected, we have $\pi^2(S) = \int_{S} \D\pi^2 > 0$. Now by Fubini's theorem we have $\int_{\Omega} I(x) \D\pi(x) = \int_{\Omega^2} W(x,y) \D\pi^2 > 0$, so we have $\pi(J) > 0$. Let $x\in J$. Since $I(x)>0$, we have $\pi(S_x)>0$. Suppose that $\pi(S_x^{+})<1$. Since $W$ is connected, we have $\pi^2(P_x)>0$. By Fubini's theorem we have
\[
  \int_{P} \D\pi^3 = \int_{J} \int_{P_x} \D\pi^2(y,z) \D\pi(x) > 0,
\]
so we have $\pi^3(P)>0$. Since $f(x,z,y)>0$ for all $(x,y,z) \in P$, we have $\int_{\Omega^3} f(x,z,y) \D\pi^3 > 0$. This contradicts~\eqref{eq:cherry-density-zero}, so we have $\pi(S_x) = \pi(S_x^{+}) = 1$.
	
Now suppose that $\pi^2(\Omega^2 \setminus S) > 0$. Since $\pi(S_x^2) = 1$, we have $\pi^2(S_x^2 \setminus S) > 0$. By Fubini's theorem we have
\[
  \int_{Q} \D\pi^3 = \int_{J} \int_{S_x^2 \setminus S}\D\pi^2(y,z) \D\pi(x) > 0,
\]
so we have $\pi^3(Q)>0$.
	
Since $f(y,z,x) > 0$ for all $(x,y,z) \in Q$, we have $\int_{Q} f(y,z,x) \D\pi^3 > 0$. This contradicts~\eqref{eq:cherry-density-zero}, so $\pi^2(S)=1$. Now since $W(x,y) \in \{0,1\}$ a.e.\ on $\Omega^2$, in fact we have $W = 1$.
\end{proof}

\section{Removal flip processes}\label{sec:RemovalProces}
In a removal flip process, we have a graph $F\in\lgr{k}$ which we remove if it is contained in the drawn graph while drawing an $F$-free graph leads to an idle step. A formal definition follows.

\begin{defi}\label{def:removal}
Let $k\in\N$ and $F\in\lgr{k}$. A rule $\rul$ of order $k$ is \emph{$F$-removal} if for each $H\in\lgr{k}$ which contains $F$ as a labelled subgraph we have $\rul_{H,H-F}=1$ while for each $H$ which does not contain $F$ as a labelled subgraph we have $\rul_{H,H}=1$. Of course, all the remaining entries of the matrix $\rul$ are~$0$.
\end{defi}

\begin{remark}
Observe that Definition~\ref{def:removal} is concerned with the removal of \emph{non-induced} copies of $F$. While we could define even flip processes removing induced copies, such a concept would not have the nice properties we study here. For example, the dual\footnote{The \emph{dual} of a flip process $\rul$ is a flip process $\rul^*$ of the same order in which $\rul^*_{F,H}$ is defined using the graph complements $\overline{F}$ and $\overline{H}$, $\rul^*_{F,H}:=\rul_{\overline{F},\overline{H}}$.} to the flip process from Remark~\ref{rem:converseMonotoneSink} shows that for such flip processes we do not have Theorem~\ref{thm:removalprocSinkLimit}.
\end{remark}

\begin{remark}
Consider Definition~\ref{def:removal} for the labelled two-edge path $F$ with $E(F)=\{12,23\}$. Observe that only in one out of three samples of $P_2$ will that path actually be removed; if a triangle is sampled, then two of its three edges are removed. While this may seem to be at odds with the spirit of $P_2$-removal, this is not actually the case; indeed, eventually all copies of $P_2$ are removed.
\end{remark}

\begin{remark}
It is not entirely apparent how to define removal rules which remove a family of graphs. For example, how should we define a removal rule of order~$5$ which removes the $4$-leaf star and the pentagon at the same time? One possibility would be to remove all edges that belong to some $4$-leaf star or to some pentagon. We leave it to the reader whether they find this definition natural and appealing.
\end{remark}

Given a graph $F\in\lgr{k}$ we say that a graphon $W$ is \emph{$F$-free} if $t(F,W)=0$.
In the rest of this section we write $(i,j)\in E(F)$ for any \emph{ordered} pair of vertices that forms an edge of a graph $F$. We substitute into~\eqref{eq:velocity} and see that for the $F$-removal rule $\rul$, we have
\begin{equation} \label{eq:removal_vel}
	\vel[\rul]W(x,y) = \sum_{(i,j) \in E(F)} \sum_{G \in \lgr{k} : F \subseteq G} \tindr{x,y}(G^{i,j},W) \cdot (-1) = - \sum_{(i,j) \in E(F)} \tr{x,y}(F^{(i,j)}, W).
\end{equation}

Most interesting features about removal flip processes concern their destinations and the way they are approached. Observe that a removal flip process is non-increasing and thus by (the dual version of) Proposition~\ref{prop:monotoneflipprocess}\ref{en:monotone1} each graphon has a destination. We have the following fact about these destinations.

\begin{fact}\label{fact:removaldests}
  If $\rul$ is the $F$-removal rule and $e(F) \ge 1$, then 
  \[
    \left\{ \dest_\rul(W) : W \in \Gra \right\} = \left\{ W \in \Gra: W \text{ is $F$-free} \right\}.
  \]
\end{fact}
\begin{proof}
Indeed, if $W$ is not $F$-free, then $\vel W\neq 0$ and Proposition~\ref{prop:dests} tells us that $W$ is not a destination.
The converse implication is clear.
\end{proof}

We give a useful identity. For the sake of brevity, we shall denote $\traj{\infty}{W}= \dest(W)$.
\begin{lem}\label{lem:removalrectangle}
	Let $k\in\N$, $F\in\lgr{k}$ and $\rul$ be the $F$-removal rule. For every graphon $W$, every $T_1\in (-\age{W},\infty)$ and $T_2\in (T_1, \infty]$ and every pair of sets $S_1,S_2\subset \Omega$ we have
	\begin{align*}
	\int_{S_1\times S_2} \left( \traj{T_1}{W}-\traj{T_2}{W} \right)
	= \sum_{(i,j)\in E(F)} \int_{t=T_1}^{T_2} \int_{(x,y)\in S_1\times S_2} t^{x,y}(F^{(i,j)},\traj{t}{W}) \;.
	\end{align*}
\end{lem}
\begin{proof}
	By Remark~\ref{rem:integral_form}, the left-hand side of the equation is
	\[\traj{T_1}{W}-\traj{T_2}{W}=\int_{t=T_1}^{T_2}(-\vel\traj{t}{W}) \;.\]
	Therefore, except when $T_2 = \infty$, the Lemma follows by applying Fubini's Theorem together with \eqref{eq:removal_vel}. For $T_2 = \infty$ we have	
	\[\int_{S_1\times S_2}\lim_{T\rightarrow \infty}\int_{T_1}^{T}(-\vel\traj{t}{W}) =\lim_{T\rightarrow \infty}\int_{S_1\times S_2}\int_{T_1}^{T}(-\vel\traj{t}{W}), \]
	by the Monotone Convergence Theorem (since the velocity is always non-positive for a removal rule). Therefore, we obtain the desired outcome for $T_2 = \infty$ by applying \eqref{eq:removal_vel} and Fubini's Theorem inside the limit on the right-hand side of the equation above. 
	\end{proof}
\subsection{Destination of a limit}

Our next theorem says that the destination is continuous for removal rules. Recall that in Remark~\ref{rem:converseMonotoneSink} we saw that this feature does not hold in the more general class of monotone rules.

\begin{thm}\label{thm:removalprocSinkLimit}
  Let $k\in\N$, $F\in\lgr{k}$ and $\rul$ be an $F$-removal rule. Suppose that $e(F)>0$. Then the function $U \mapsto \dest_\rul(U)$ is continuous as a function from $ \left( \Gra, \cutnd \right)$ to itself.
\end{thm}

Theorem~\ref{thm:removalprocSinkLimit} is an important stepping stone towards our proof of Theorem~\ref{thm:speedremoval}, which shows that convergence time is finite and uniformly bounded. The proof of Theorem~\ref{thm:removalprocSinkLimit} is quite involved. In particular, it uses Szemer\'edi's regularity lemma, but the advantage is that then the proof of Theorem~\ref{thm:speedremoval} is short and conceptually simple. On the other hand, Theorem~\ref{thm:speedremoval} is very similar in spirit to the classical removal lemma. In particular, we ask whether Theorem~\ref{thm:speedremoval} could be proven using the removal lemma instead of the Szemer\'edi's regularity lemma.

\begin{proof}
The statement is obvious when $F$ has at most~2 vertices. So, let us assume that $v(F)\ge 3$.

  Given a graphon $W$, we write $W_t = \traj{t}W$ and $W_\infty = \dest(W)$. We will show that for every $\gamma>0$ (without loss of generality, $\gamma < 1$) there exists $\delta$ such that if $\cutnd(W,U)<\delta$ then $\cutnd(W_\infty,U_\infty ) < \gamma$.
With foresight of the cooking process, we set
\[
  \beta := \frac{\gamma}{8k^2}, \quad \alpha := \frac{\beta^{e(F)}}{2k^{k}}.
\]
	Let $M_0$ be given by Theorem~\ref{thm:RL} with error parameter $\alpha$.
Let 	
$\eps := \beta^{e(F)}/ (2k^2 ( M_0 )^{k})$.
Fix an arbitrary graphon $U$ and pick $T$ such that
	\begin{equation}\label{eq:5TWO}
	\cutnd(U_T, U_\infty ) < \frac{\eps}{2}\;.
	\end{equation}
Let $\delta := \eps/(2\exp(C_k T))$, where $C_k$ comes from Theorem~\ref{thm:LipschTime}.
	
From now on we fix a graphon $W$ satisfying $\cutnd(W,U)<\delta$. Theorem \ref{thm:LipschTime} implies that
	\[
	  \cutnd(W_T, U_T) \le \exp(C_k T)\cdot\cutnd(W, U) < \frac{\eps}{2}\;,
	\]
	which we combine with \eqref{eq:5TWO} to get that
	\begin{equation}\label{eq:8TWO}
	  \cutnd(W_T, U_\infty) < \eps \;.
	\end{equation}
	We will show that
	\begin{equation}
	  \label{eq:dist_to_show}
	  f := \cutnd (W_T, W_\infty) < \gamma/2
	\end{equation}
	and combine with \eqref{eq:8TWO} to finish the proof.

	First notice that since $U_\infty $ is a destination, Fact~\ref{fact:removaldests} implies that $t(F,U_\infty )=0$ and therefore Lemma~\ref{lem:countinglemma} yields
	\begin{equation}\label{eq:XtfewtrianglesTWO}
	  t(F,W_T) < k^2\cutnd( W_T, U_\infty ) \leByRef{eq:8TWO} k^2 \eps\;.
	\end{equation}
	
	The strategy is to obtain a contradiction by showing that if $W_T$ is not too close to its destination, then it has substantial density $t(F, W_T)$. Since the $F$-removal rule is non-increasing, we have $W_T \ge W_\infty$ and in particular
      \[
	f = \cutnd(W_T, W_\infty) = \int_{\Omega^2} (W_T - W_\infty) \;.
      \]
      We apply Theorem~\ref{thm:RL} to find an $\alpha$-regular partition $\cp$ of the graphon $W_T-W_\infty$ with $M \le M_0$ parts.
	
      Given $\mathbf{S}=(S_1,\ldots,S_k) \in \cp^{k}$, recall Definition~\ref{def:densities} and set
	\[	
	w_{\mathbf{S}} : =\int_{t=T}^\infty t^{\mathbf{S}}(F,W_t) \;.
          \]
	From Lemma~\ref{lem:removalrectangle} we get that
	\begin{align}
	\label{eq:claim1TWO}
	f &= \int_{\Omega^2} \left( W_T - W_\infty \right) = \sum_{(i,j)\in E(F)} \int_{t = T}^\infty \int_{(x,y) \in \Omega^2} t^{x,y}(F^{(i,j)}, W_t)\\
	& = \sum_{(i,j) \in E(F)} \int_{t = T}^\infty \sum_{\mathbf{S} \in \cp^k} t^{\mathbf{S}}(F, W_t) = \sum_{(i,j)\in E(F)} \sum_{\mathbf{S}\in\cp^k} w_{\mathbf{S}} \;.
	\end{align}
	By~\eqref{eq:claim1TWO}, there exists $(i,j)\in E(F)$ such that
 \begin{equation}\label{eq:bridge}
 \sum_{\mathbf{S} \in \cp^k}
 w_{\mathbf{S}}\ge \frac{f}{k^2}\;.
 \end{equation}

 Let $\cb_2 \subset \cp^2$ be the set of (ordered) pairs that are either not $\alpha$-regular in the graphon $W_T-W_\infty$ or have density below $\beta = \frac{\gamma}{8k^2}$.  Let $\cb_k \subset \cp^k$ be the set of $k$-tuples $\mathbf{S}$ such that there exists $(i,j) \in E(F)$ such that $(S_i,S_j) \in \cb_2$.
 We have
\begin{equation*}
 \begin{split}
   \sum_{\mathbf{S}\in \cb_k} w_{\mathbf{S}}& \le
   \sum_{(S'_1,S'_2)\in \cb_2}\sum_{(i,j)\in E(F)} \sum_{\mathbf{S} : S_i = S'_1, S_j = S'_2 }w_{\mathbf{S}} \\
   & = \sum_{(S'_1,S'_2)\in \cb_2}\sum_{(i,j)\in E(F)} \int_{t = T}^\infty \int_{(x,y)\in S_1' \times S_2'} t^{x,y}(F^{(i,j)}, W_t) \\
   \justify{Lemma \ref{lem:removalrectangle}}&= \sum_{(S'_1,S'_2)\in \cb_2} \int_{S'_1\times S'_2} (W_T-W_\infty) \;.
 \end{split}
\end{equation*}
Since, by $\alpha$-regularity of $\mathcal{P}$, the total measure of rectangles $S_1' \times S_2'$ that are not $\alpha$-regular is at most $\alpha M^2 \cdot M^{-2} = \alpha$, and the integral over the remaining rectangles is at most $\beta$, we obtain
\begin{equation}
 \label{eq:fHR}
   \sum_{\mathbf{S}\in \cb_k} w_{\mathbf{S}} \le
    \alpha + \beta \le \frac{\gamma}{4k^2}\;.
\end{equation}
We are now in a position to complete the proof by showing~\eqref{eq:dist_to_show}. Suppose for contradiction that $f \ge \gamma/2$. Then from~\eqref{eq:bridge} and~\eqref{eq:fHR} we have
\[
  \sum_{\mathbf{S}\in\cp^k\setminus \cb_k } w_{\mathbf{S}} \ge \frac{\gamma}{4k^2} > 0\;.
\]
In particular, that means that $\cp^k\setminus \cb_k \ne \emptyset$. Let us take an arbitrary $\mathbf{S}\in\cp^k\setminus \cb_k$. By the definition of $\cb_k$, for each $(a,b)\in E(F)$ we have, in the graphon $W_T - W_\infty$, that $(S_a,S_b)$ is an $\alpha$-regular pair of density at least $\beta$. Lemma~\ref{lem:RLcounting} gives that
\begin{equation*}
  t(F,W_T)\ge t^{\mathbf{S}}(F,W_T-W_\infty) \ge \left( \beta^{e(F)} - \alpha k^k\right) \left( \frac{1}{M_0} \right)^k \ge \frac{\beta^{e(F)}}{2 M_0^{k}} \ge k^2 \eps \;,
\end{equation*}
a contradiction to~\eqref{eq:XtfewtrianglesTWO}. Hence, we conclude that $f<\gamma/2$, as was needed.
\end{proof}

\begin{remark}\label{rem:removalVSregularity}
One of the main steps in the proof above is similar to the standard deduction of the removal lemma from Szemer\'edi's regularity lemma. That is, substantial density of $F$ in $W_T$ is deduced from a single copy of $F$ in the cluster graph (`$\cp^k\setminus \cb_k \ne \emptyset$'). We think that finding a way to instead directly apply the removal lemma is an interesting problem in its own right. Note that such a direct application of the removal lemma would likely yield an improvement in the implicit epsilon--delta quantification of continuity in Theorem~\ref{thm:removalprocSinkLimit} since it would allow the use of the bounds famously obtained by Fox~\cite{Fox:Removal}.
\end{remark}

\subsection{Speed of removal flip processes}
Generally, removal flip processes may converge to the destination very slowly. For example, consider the triangle removal flip process. A well known construction of Behrend (see \cite[Corollary 2.5.2]{Zhao:GraphTheoryAdd}) implies that for every $n\in \N$ there exists a graphon $L_m$ with $t(K_3,L_m)=O(1/m^2)$, $\|L_m\|_1=\exp(-\Theta(\sqrt{\log m}))$, and whose destination is $\dest(L_m)=0$. Obviously, for the velocity of the triangle removal flip process we have, for each $T\ge 0$, that $\|\vel \traj{T}{L_m}\|_1 \le \|\vel L_m\|_1 = 6 t(K_3, L_m) = O(1/m^2)$. This example shows that the convergence time\footnote{Due to monotonicity, $\tau_\rul^- = \tau_\rul^+$ for removal processes.} $\tau_\rul^-(\delta)$ cannot be bounded from above by any polynomial in $1/\delta$.
So far the reasoning for the slow speed of convergence of the triangle removal flip process\footnote{which we make formal and more general in Proposition~\ref{prop:removalBipVSNonbip}} was based on lower bound constructions for the triangle removal lemma. But a priori we cannot rule out even the case that $\tau_\rul^-(\delta)=\infty$. This is because the removal lemma says `if the graph is far from all triangle-free graphs then it has many triangles' and not `if the graph is far from its destination then it has many triangles'.

Our result shows that the convergence times of each removal rule are finite and uniformly bounded. Theorem~\ref{thm:removalprocSinkLimit} plays a key role in the proof.
\begin{thm}\label{thm:speedremoval}
  Let $k\in\N$, $F\in\lgr{k}$ and $\rul$ be the $F$-removal rule. Suppose that $e(F)>0$. Then for every $\eps>0$ there exists $T \ge 0$ so that for every graphon $W$ we have $\|\traj{T}{W} - \dest(W)\|_1 < \eps$.
\end{thm}
\begin{proof}
Suppose that this does not hold. Then there exists $\eps>0$ and a sequence of graphons $(W_n)_{n\in\N}$ so that for every $n\in \N$, we have
\begin{equation}\label{eq:spor}
\|\traj{n}{W_n}-\dest(W_n)\|_1\ge \eps\;.
\end{equation}
By the Lov\'asz--Szegedy compactness theorem (see \cite[Theorem 9.23]{Lovasz2012}) there exist indices $n_1<n_2<\cdots$, measure preserving bijections $\psi_1,\psi_2,\ldots$, and a graphon $U$ so that writing $U_i:=W_{n_i}^{\psi_i}$ we have
\begin{equation}\label{eq:2}
  \cutn{U_i-U}\rightarrow 0\;.
\end{equation}
By Theorem~\ref{thm:removalprocSinkLimit}, we have
\begin{equation}\label{eq:3}
  \cutn{\dest(U_i) - \dest(U)} \rightarrow 0\;.
\end{equation}

Take $T$ such that $\cutn{\traj{T}{U} - \dest(U)} < \eps/4$. Let $C_k$ be as in Theorem~\ref{thm:LipschTime}. By~\eqref{eq:2} and~\eqref{eq:3} we can choose $\ell\ge T$ such that 
\begin{equation}\label{eq:Gihs}
  \cutn{U_\ell-U} < \eps/(4\exp(C_kT)) \quad \text{and} \quad \cutn{\dest(U_\ell) - \dest(U)}<\eps/4. 
\end{equation}
Since $n_\ell \ge \ell \ge T$, we have
\begin{align*}
\Lone{\traj{n_\ell} U_\ell - \dest(U_\ell)} &= \int_{\Omega^2} \left( \traj{n_\ell}{U_\ell}-\dest(U_\ell) \right) \\
  & \le \int_{\Omega^2} \left( \traj{T}{U_\ell}-\dest(U_\ell) \right) \\
  & \le \cutn{\traj{T}{U_\ell}-\traj{T}{U}} + \cutn{ \traj{T}{U} - \dest(U) } + \cutn{\dest(U_\ell) - \dest(U)}\\
  \justify{\eqref{eq:Gihs} and Theorem~\ref{thm:LipschTime}} &< \frac{\eps}{4} + \frac{\eps}{4} + \frac{\eps}{4} < \eps\;.
\end{align*}
Noting that $\traj{n_\ell}U_\ell = (\traj{n_\ell}W_{n_\ell})^{\psi_{\ell}}$ and $\dest(U_\ell) = (\dest(W_{n_\ell}))^{\psi_\ell}$ (see Section~\ref{FIRSTPAPER.ssec:repermuting} in \cite{Flip1}), we obtain a contradiction to~\eqref{eq:spor}.
\end{proof}

We do not obtain any explicit bounds on the convergence time $T$ in Theorem~\ref{thm:speedremoval}. In lieu of the Lov\'asz--Szegedy compactness theorem, one could probably use the machinery of a regularity lemma (the Frieze-Kannan version seems to be sufficient) in the proof. Further, to obtain an explicit bound in Theorem~\ref{thm:speedremoval}, one would make the $\eps$--$\delta$ quantification of continuity in Theorem~\ref{thm:removalprocSinkLimit} explicit as well. Recall that Theorem~\ref{thm:removalprocSinkLimit} uses Szemer\'edi's regularity lemma (Theorem~\ref{thm:RL}) with its infamous tower-type dependency. Putting them together, we would likely get that $T\le \mathsf{tower}((1/\eps)^{O(1)})$; even if we successfully replace Szemer\'edi's regularity lemma with the removal lemma as mentioned in Remark~\ref{rem:removalVSregularity}, we would still get only $T\le \mathsf{tower}(O(\log(1/\eps)))$.
In our next proposition, however, we prove that a polynomial upper bound exists if and only if $F$ is bipartite. To this end, some preparatory notation is needed.

Suppose that $F$ is a bipartite graph. Let $m_F$ be the smallest number such that for each graphon $W$ we have 
\begin{equation}
  \label{eq:Sido_weak}
  t(F,W)\ge \|W\|_1^{m_F} \;. 
\end{equation}
We call $m_F$ the \emph{Sidorenko power of $F$}. It is easy to show that $m_F$ exists and $m_F\in[e(F),(\frac{v(F)}2)^2]$. The Sidorenko conjecture (see \cite{Lovasz2012}) asserts that $m_F=e(F)$ for every bipartite $F$.
\begin{prop}\label{prop:removalBipVSNonbip}
  Let $k\in\N$, $F\in\lgr{k}$ and $\rul$ be the $F$-removal rule. Suppose that $e(F)>0$. Then the following hold.
  \begin{romenumerate}
  \item \label{en:dest_zero} If $F$ is bipartite, then the destination of every graphon is the constant-$0$ graphon.
  \item \label{en:poly_speed} Suppose that $F$ is bipartite and its Sidorenko power is $m_F$. Then Theorem~\ref{thm:speedremoval} holds with $T=\eps^{-m_F+1}$ when $e(F) \ge 2$ and with $T = \ln (1/\eps)$ when $e(F) = 1$.
  \item \label{en:no_poly} If $F$ is not bipartite then no polynomial bound on $T$ in terms of $\eps$ exists, that is, for every $M$ there is $\eps_0>0$ so that for all $\eps \in (0, \eps_0]$ we have $\tau_\rul^-(\eps) \ge \eps^{-M}$.
      \end{romenumerate}
\end{prop}
\begin{proof}
First suppose that $F$ is bipartite. Let $W$ be a graphon. By~\eqref{eq:removal_vel} and~\eqref{eq:Sido_weak} we have
\begin{equation*}\label{eq:EVtss}
  \Lone{ \vel W } = (k)_2 \cdot t(F,W) \ge \Lone{ W }^{m_F}\;.
\end{equation*}
Then Proposition~\ref{prop:dests} tells us any destination is constant-$0$, proving \ref{en:dest_zero}.
	
We turn to \ref{en:poly_speed}. Let us consider the case $e(F) \ge 2$, which entails $m_F\ge 2$; we shall return to the case $e(F) = 1$ later. The required conclusion is trivial for the constant-$0$ graphon, so consider a graphon $W$ which is not constant-$0$.
We claim that the (continuous) function $f(t):=\Lone { \traj{t}{W} }$ satisfies 
\begin{equation} \label{eq:f_der_low}
  -f'(t) \ge f(t)^{m_F}.
\end{equation}
Indeed, Lemma~\ref{lem:removalrectangle} and~\eqref{eq:Sido_weak} imply
\begin{align*}
  f(t) - f(t + \eps) &= \int_{\Omega^2} \left( \traj{t} W - \traj{t+\eps} W  \right) = \sum_{(i,j)\in E(F)} \int_{\tau = t}^{t + \eps} \int_{(x,y)\in \Omega^2} t^{x,y}(F^{(i,j)},\traj{\tau}{W}) \\
  &\geByRef{eq:Sido_weak} (k)_2\int_t^{t + \eps} \Lone{ \traj{\tau}W }^{m_F} \ge (k)_2 \eps f(t + \eps)^{m_F} \ge \eps f(t + \eps)^{m_F}\;.
\end{align*}
Dividing by $\eps$ and taking the limit $\eps \to 0$ we obtain \eqref{eq:f_der_low}, proving our claim. Now~\eqref{eq:f_der_low} implies that the function $g(t) := f(t)^{-(m_F - 1)}$ satisfies 
\[
  g'(t) = - (m_F - 1) f(t)^{-m_F} f'(t) \ge m_F - 1,
\]
where we used the fact that $f$ takes positive values. Consequently, we have $g(t) \ge g(0) + (m_F - 1)t$, and by substituting $g$ we obtain
\[
  f(t) \le \left((m_F-1)t+f(0)^{1-m_F}\right)^{-\frac1{m_F-1}} \le  t^{-\frac1{m_F-1}}\;,
\]
and the desired conclusion in Theorem~\ref{thm:speedremoval} follows. Similarly, for the case $e(F) = 1$ we have the differential inequality $-f'(t) \ge f(t)$ with auxiliary function $g(t) = -\ln f(t)$. This gives $f(t) \le f(0) \exp(-t)$, which then implies the desired conclusion.
	
	To deduce \ref{en:no_poly}, we recall one of the main results of~\cite{Alon:TestingSubgraphs}. Namely, Lemma~3.4 in~\cite{Alon:TestingSubgraphs} asserts that if $F$ is not bipartite, then there exists $c>0$ such that for all sufficiently small $\alpha>0$ and all sufficiently large integers $n$ there exists an $n$-vertex graph $G$ that is $\alpha$-far from being $F$-free and contains at most $(\frac\alpha{c})^{c\log(\frac{c}{\alpha})} n^{v(F)}$ copies of $F$. Here, \emph{$\alpha$-far} means that more than $\alpha n^2$ edges need to be removed from $G$ so that it becomes $F$-free. Since the number of non-injective homomorphisms of $F$ into $G$ is $O(n^{v(F) - 1})$, by taking $n$ large enough (depending on $\alpha$), we can assume that the the homomorphism density $t(F,G)$ is at most $2(\frac\alpha{c})^{c\log(\frac{c}{\alpha})}$. Let $W$ be a graphon representation of $G$. Recalling that $t(F, W) = t(F, G)$, we have $
t(F,W) \le 2 \left(\frac\alpha{c}\right)^{c\log(\frac{c}\alpha)}$.
	Also, note that since $t(F, \traj{\cdot}{W})$ is nonincreasing for the $F$-removal flip process, by Lemma~\ref{lem:removalrectangle} we have
	\begin{equation}
	  \label{eq:Lone_diff_upper}
	  \Lone{W - \traj{t}{W}} = (v(F))_2 \int_{x = 0}^t t(F, \traj{x}{W}) \le t(v(F))_2 \cdot t(F, W)\le t(v(F))_2\cdot 2 \left(\frac\alpha{c}\right)^{c\log(\frac{c}\alpha)}.
	\end{equation}
	The notion of being $\alpha$-far translates in a straightforward way into a graphon representation. In this translation, the $L^1$-norm is the graphon counterpart to measuring the number of edges removed. Let us use this intuition to show that
\begin{equation} \label{eq:destalpha}
\|W-\dest(W)\|_1\ge \alpha\;.
\end{equation}
Indeed, it follows from Proposition~\ref{FIRSTPAPER.prop:twinsstay} of~\cite{Flip1} that $\dest(W)$ is a step graphon on steps $(\Omega_v:v\in V(G))$ corresponding to the representation of $G$. Let $G'\subset G$ be the subgraph of $G$ containing edges in the support of $\dest(W)$, that is, $ij\in E(G')$ if and only if $\dest(W){\restriction_{\Omega_i\times \Omega_j}}>0$. Obviously, $\frac{e(G)-e(G')}{n^2}\le \|W-\dest(W)\|_1$. At the same time, $\dest(W)$ is $F$-free, and therefore so is $G'$. Since $G$ is $\alpha$-far from $F$-freeness, we have $\frac{e(G)-e(G')}{n^2}>\alpha$. This allows us to deduce~\eqref{eq:destalpha}. 

Thus, for every $t \ge 0$ we have
\[
  \Lone{\traj{t}{W} - \dest(W)} \ge 
  \Lone{W - \dest(W)} - \Lone{W - \traj{t}{W}} 
  \geBy{\eqref{eq:destalpha}, \eqref{eq:Lone_diff_upper}}
  \alpha - t \cdot 2(v(F))_2 \left(\frac\alpha{c}\right)^{c\log(\frac{c}\alpha)}\;.
\]

Given $M > 0$, taking $\alpha>0$ small enough, we have $2(v(F))_2(\tfrac{\alpha}{c})^{c\log(\frac{c}\alpha)} \le (\alpha/2)^{M + 1}$. Setting $\eps=\alpha/2$ and $t = \eps^{-M}$ we have $\Lone {\traj{t}{W}-\dest(W)} \ge \eps$, whence $\tau_\rul^-(\eps) \ge \eps^{-M}$. This completes the proof of \ref{en:no_poly}.
\end{proof}

\subsection{Graphons with the same destination}
Fact~\ref{fact:removaldests} (see also Proposition~\ref{prop:removalBipVSNonbip}) tells us that for non-bipartite graphs $F$, the $F$-removal rule has a nontrivial set of destinations. An intriguing problem for any such $F$ is to determine equivalence classes of graphons with the same destination. A particularly important equivalence class is the class of graphons whose destination is constant-$0$. We do not have even a guess what the answer might be. 
\begin{qu}\label{qu:zerosinksFremoval}
	Suppose that $F$ is a fixed graph. Determine all graphons whose destination in the $F$-removal flip process is the constant-0 graphon.
\end{qu}
In this section we consider the simplest nontrivial example, $F=K_3$, and we study the destinations of step graphons $\left\{ W_{\alpha,\beta} : \alpha, \beta\in [0,1] \right\}$ with two steps $\Omega_1$ and $\Omega_2$ of equal size such that $W_{\alpha, \beta} = \alpha$ on $\Omega_1^2 \cup \Omega_2^2$ and $W_{\alpha, \beta} =  \beta$ elsewhere. It is easy to see (using \cite[Corollary 5.6 and Remark 3.8]{Flip1}) that the trajectory starting at $W_{\alpha,\beta}$ stays a step graphon with respect to the partition $\Omega_1 \cup \Omega_2$ and becomes a constant graphon if and only if $\alpha = \beta$. An easy symmetry argument (using \cite[Remark 3.8]{Flip1}, say), also implies that the value on $\Omega_1^2$ stays equal to the value on $\Omega_2^2$.
To sum up, for all $\alpha,\beta \in [0,1]$ there are some functions $a, b : [0, \infty) \to [0,1]$ such that we have
\begin{align}
  & \alpha = \beta \text{ implies }  a(t) = b(t) \text{ for every } t \ge 0\;, \label{eq:always_equal} \\
  & \alpha \ne \beta \text{ implies }  a(t) \ne b(t) \text{ for every } t \ge 0\;, \label{eq:never_equal}
\end{align}
and for all $t \ge 0$ we have
\begin{equation}\label{eq:functionsAB}
\traj{t}W_{\alpha, \beta} = W_{a(t),b(t)}\;.
\end{equation}

It is very intuitive that $\dest(W_{\alpha, \beta}) = W_{0, \gamma}$, with $\gamma > 0$ whenever $\alpha/\beta$ is small enough and $\gamma = 0$ when $\alpha/\beta$ is large enough. The following result says that the critical point is $\alpha=\beta$.
\begin{prop} \label{prop:rem_tri_dest_bip}
  Consider the $K_3$-removal flip process. For a step graphon $W_{\alpha, \beta}$ as defined above we have
  \[
    \dest(W_{\alpha, \beta}) =\begin{cases}
    W_{0,\beta\exp [-2\alpha/(\beta - \alpha)]}  &\text{ if } \alpha < \beta\;, \\
    W_{0,0},  &\text{ if } \alpha \ge \beta \;.
      \end{cases}
  \]
\end{prop}
\begin{proof}
Since $\vel W(x,y) = -6\tr{x,y}(K_3^{\bullet \bullet}, W)$, it is easy to check that the functions $a$ and $b$ from~\eqref{eq:functionsAB} satisfy the system of differential equations
\begin{align*}
  a' &= -a^3/2 - ab^2/2, \\
  b' &= -ab^2
\end{align*}
with the initial condition $a(0) = \alpha, b(0) = \beta$.
Observe that since the functions $a$ and $b$ take non-negative values, it follows easily from the system of differential equations that in the case $\alpha = 0$ we have $a \equiv 0$ and $b \equiv \beta$. 
A similar argument tells us that in the case $\beta = 0$ we get $b \equiv 0$ and $a(t) = \frac{1}{\sqrt{t + 1/\alpha^2}}$.
Now if $\alpha = \beta > 0$, then by~\eqref{eq:always_equal} we actually have a single equation $a' = - a^3$, which has an explicit solution $a(t) = b(t) = \frac{1}{\sqrt{2(t + 1/(2\alpha^2) )}}$. Hence, we may assume that $\alpha \ne \beta$, $\alpha > 0, \beta > 0$. 
This implies that $a, b$ take positive values for every $t \ge 0$ (see \cite[Proposition~\ref{FIRSTPAPER.prop:valueszeroone}]{Flip1}) and hence are strictly decreasing continuous functions. In particular $\lim_{x \to \infty} a(x)$ is nonnegative; it is actually zero, in view of the differential equation for $a'$.
Hence the inverse function $a^{-1}$ is defined on $(0, \alpha]$. Define a function $f : (0, \alpha] \to (0, 1]$ by $f(x) := b(a^{-1}(x))$. Note that
\begin{equation}
  \label{eq:two_step_dest}
  \dest(W_{\alpha,\beta}) = W_{0, \lim_{x \to 0} f(x)}.
\end{equation}

In view of~\eqref{eq:two_step_dest}, it remains to determine $\lim_{x \to 0} f(x)$. To do so, we shall seek to understand the function $f$. Since $\alpha \neq \beta$, note that \eqref{eq:never_equal} and continuity implies that
\begin{equation}
  \label{eq:f_dichotomy}
  \text{either}\quad f(x) < x \ \forall x \in (0, \alpha] \quad\text{or}\quad f(x) > x \ \forall x \in (0, \alpha]
\end{equation}
We have
\begin{align*}
  f'(x) &= \frac{ b' \left( a^{-1}(x) \right) }{ a' \left( a^{-1}(x) \right) } = - \frac{ a \left( a^{-1}(x) \right)  b(a^{-1}(x)) ^2 }{ a \left( a^{-1}(x) \right)^3/2 + a \left( a^{-1}(x) \right) b \left( a^{-1}(x) \right)^2/2} \\
  &= \frac{2xf(x)^2}{x^3 + xf(x)^2} = \frac{2}{(x/f(x))^2 + 1} \;.
\end{align*}
Hence $f$ satisfies a differential equation $f' = 2/(x^2/f^2 + 1)$ for $x \in (0, \alpha]$ with initial condition $f(\alpha) = \beta$. Denoting $v(x) = f(x)/x$, we have $f' = v'x + v$, whence
\begin{align}
  \label{eq:xvprime} xv' & = f' - v = \frac{2}{v^{-2} + 1} - v = \frac{-v (v - 1)^2}{1 + v^2} \;, \\
  \notag \frac{v'}{v} & = -\frac{v^2 - 2v + 1}{x(1 + v^2)} = -\frac{1}{x} + \frac{2v}{x(1 + v^2)} \;.
\end{align}
Since $(\ln v)' = v'/v$ and $ \left( 2/(v-1) \right)' = -2v'/(v-1)^{2}$, we get
\begin{align*}
  \ln v(x)  &= \ln v(\alpha) + \int_\alpha^x \left( \frac{2v}{x(1 + v^2)} - \frac{1}{x}  \right) \D x \eqByRef{eq:xvprime} \ln \frac{f(\alpha)}{\alpha} + \ln \frac{\alpha}{x} + \int_\alpha^x \frac{-2v'}{(v - 1)^2} \\
  \justify{$f(\alpha) = \beta$} &=\ln \frac{\beta}{x} +  \frac{2}{v(x) - 1} - \frac{2}{v(\alpha)-1}.
\end{align*}

Since $v(x) = f(x) / x$ and $v(\alpha) = \beta/\alpha$, it follows that any solution $f$ satisfies equation
\begin{equation}
\label{eq:new_diffeq}
  \ln \frac{f(x)}{x} - \frac{2}{f(x)/x - 1} = \ln \frac{\beta}{x} + \frac{2}{1 - \beta/\alpha}.
\end{equation}
So, writing $y = f(x)/x - 1$, $F(y) := \ln(1 + y) - 2/y$ and defining a constant
\[
  c_0 := \ln \beta + \frac{2}{1 - \beta/\alpha},
\]
equation \eqref{eq:new_diffeq} reads
\[
  F(y) = c_0 - \ln x.
\]
Elementary calculus shows that $F$ is strictly increasing and continuous on $(-1,0)$ and on $(0, \infty)$ and undefined elsewhere (which is not a problem, since by \eqref{eq:f_dichotomy} we have $y \in (-1,0) \cup (0, \infty)$ for $x \in (0, \alpha]$). Moreover, $F( (-1,0) ) = F( (0, \infty) ) = \R$.
Hence the inverse of $F$ consists of two increasing bijections: $F^{-1}_- : \R \to (-1, 0)$ and $F^{-1}_+ : \R \to (0, \infty)$. By \eqref{eq:f_dichotomy}, and noting that $f(\alpha) = \beta$ we get that if $\alpha < \beta$, then $y(\alpha) = \beta/\alpha - 1 > 0$ and thus $y > 0$ for all $x$; similarly, if $\alpha > \beta$, then $y < 0$ for all $x$. Consequently
\begin{equation}
\label{eq:solutions}
    f(x) =  \begin{cases}
 x\left(1 + F^{-1}_-(c_0 - \ln x) \right), \quad \text{if } \alpha > \beta, \\
x\left(1 + F^{-1}_+(c_0 - \ln x) \right), \quad \text{if } \alpha < \beta.
  \end{cases}
\end{equation}
We shall now determine $\lim_{x \downto 0} f(x)$; in view of \eqref{eq:two_step_dest}, this will complete the proof. It clearly equals zero when $\alpha > \beta$, since $F^{-1}_-$ is bounded. In the case $\alpha < \beta$, we have that $F_+^{-1}(x) \to \infty$, as $x \to \infty$, so we estimate $F^{-1}_+(x)$ for large $x$ by choosing an arbitrary constant $\eps > 0$ and using inequalities
\[
  \ln (1 + y) - \eps =: L_\eps(y) \le F(y) \le L_0(y) = \ln (1 + y)
\]
that hold for large enough $y$.
We note that for large enough $x$
\[
  e^x - 1 = L_0^{-1}(x) \le F^{-1}_+ (x) \le L_\eps^{-1}(x) = e^{x + \eps} - 1,
\]
which implies that $xF^{-1}_+(c_0 - \ln x) \in [e^{c_0}, e^{c_0+ \eps}]$. Since $\eps > 0$ was arbitrary, we obtain that
\[
  \lim_{x \downto 0} x (1 + F^{-1}_+(c_0 - \ln x) ) = e^{c_0} = \beta \exp \left( -\frac{2\alpha}{\beta - \alpha} \right) \;,
\]
completing the proof.
\end{proof}

\subsection{Stabilization in a removal flip process}\label{ssec:removaluntilend}
In Section~\ref{ssec:extremistStabilization} we argued that extremist flip processes may take very long to stabilize. On the other hand, removal flip processes stabilize substantially faster. Indeed, an $F$-removal flip process started on an $n$-vertex graph stabilizes no later than the moment that each of its $(n)_k$-many distinct $k$-tuples of vertices is sampled at least once. Then, the solution to the classical coupon collector's problem tells us that $\Theta(n^k\log n)$ is asymptotically almost surely an upper bound on the stabilization time. We do not know whether this bound is optimal.

\begin{qu}\label{qu:stabilizatioremoval}
	For each graph $F$ with the corresponding $F$-removal flip process $\rul_F$, we seek a function $b_F:\N\rightarrow\N$ growing as slowly as possible so that \[
	\min_{G \in \lgr{n}}\prob{ \kappa_{\rul_F}(G)\le b_F(n) } \to 1, \quad \text{ as } n \to \infty \;.
	\]
\end{qu}

Since the stabilization time is clearly superquadratic, the limitations of Theorem~\ref{thm:correspondenceInformal} kick in. That means that \emph{a priori} we cannot connect the concept of destinations to the structure of a \emph{final graph}, that is, the graph after stabilization occurs. Our next theorem shows that, in the case of removal flip processes, destinations serve as good approximations for the final graphs.

\begin{thm}\label{thm:removaluntiltheend}
Given a graph $F$, for every $\eps>0$ there exists $n_0\in\N$ such that the following holds for all $n \ge n_0$. Suppose that $G^*$ is the final graph in the $F$-removal flip process started at an $n$-vertex graph $G$. Let $W_G$ and $W_{G^*}$ be the graphon representations of $G$ and $G^*$, respectively. Then with probability at least $1-\eps$ we have $\cutnd(\dest(W_G),W_{G^*})<\eps$.
\end{thm}
The proof uses several tools, including Theorem~\ref{thm:speedremoval} and the graph removal lemma, the latter of which we now recall.
\begin{prop}[Graph removal lemma]\label{prop:graphremoval}
	Given a graph $F$, for every $\alpha>0$ there exists $\beta>0$ with the following property. Suppose that $G$ is an $n$-vertex graph whose density of $F$ is less than $\beta$. Then we can remove at most $\alpha n^2$ edges from $G$ and make it $F$-free.
\end{prop}
\begin{proof}[Proof of Theorem~\ref{thm:removaluntiltheend}]
Let $\eps>0$ and $F$ be a graph. Observe that the statement is trivial when $e(F) = 0$, so we shall assume that $e(F) > 0$ and $v(F) \ge 2$. Proposition~\ref{prop:graphremoval} returns $\beta > 0$ for input $\alpha := \frac{\eps}{4e(F)}$. Set $\gamma := \min\{\alpha,\beta/k^2\}$. Theorem~\ref{thm:speedremoval} returns $T \ge 0$ for error parameter $\gamma$. Lastly, let $n_0$ be sufficiently large so that we have $1-\exp(-\Omega(n_0^2))\ge 1-\eps$ in Theorem~\ref{thm:correspondenceInformal} with error parameter $\gamma$ and input $T \ge 0$.

Suppose now $G$ is an $n$-vertex graph for $n\ge n_0$. We run the $F$-removal flip process on $G$ and obtain a sequence of graphs $G_0:=G,G_1,G_2,\ldots,G_{\kappa(G)}=:G^*$. Let $\cb$ be the event that for all $\ell\in\N\cap [0,Tn^2]$ we have $\cutnd(W_{G_\ell},\traj{\ell/n^2}{W_{G_0}})<\gamma$. Theorem~\ref{thm:correspondenceInformal} tells us that the probability of this event is at least $1-\eps$. Assume that $\cb$ holds. We shall prove that $\cutnd(\dest(W_G),W_{G^*})<\eps$.

We have $\cutnd\left(\dest(W_G),\traj{T}{W_G}\right) \le \gamma$ and $\cutnd\left(\traj{T}{W_G},W_{G_{Tn^2}}\right) \le \gamma$ by Theorem~\ref{thm:speedremoval} and Theorem~\ref{thm:correspondenceInformal} respectively. Putting them together, we obtain
\begin{equation} \label{eq:removal-stabilize-dest-T-dist}
  \cutnd\left(\dest(W_G),W_{G_{Tn^2}}\right) \le \cutnd\left(\dest(W_G),\traj{T}{W_G}\right) + \cutnd\left(\traj{T}{W_G},W_{G_{Tn^2}}\right) \le 2\gamma.
\end{equation}
Since $\dest(W_G)$ is $F$-free by Fact~\ref{fact:removaldests}, it follows from~\eqref{eq:removal-stabilize-dest-T-dist} and Lemma~\ref{lem:countinglemma} that $t(F,W_{G_{Tn^2}}) < \beta$. Then, Proposition~\ref{prop:graphremoval} tells us that we can find a set $R\subset E(G_{Tn^2})$ of at most $\frac{\eps n^2}{4e(F)}$ edges whose removal would destroy all the copies of $F$ in $G_{Tn^2}$. In particular, every copy of $F$ in $G_{Tn^2}$ has an edge in $R$. Let $X \subseteq [Tn^2,\ldots,\kappa(G)]$ be the indices of the non-idle steps. In particular, $X \neq \emptyset$ if and only if $\kappa(G)>Tn^2$. Observe that in each non-idle step indexed in $X$ a copy of $F$ in $G_{Tn^2}$ is removed, so at least one edge from $R$ is removed. Hence, we have
\[
  e(G_{Tn^2}) - e(G^*) = e(F) \cdot |X| \le e(F) \cdot |R| \le \frac{\eps n^2}{4} \;,
\]
which implies $\cutnd\left(W_{G_{Tn^2}},W_{G^*}\right) \le \eps/4$. Combining this with~\eqref{eq:removal-stabilize-dest-T-dist}, we obtain
\[
  \cutnd\left(\dest(W_G),W_{G^*}\right)
  \le \cutnd\left(\dest(W_G),W_{G_{Tn^2}}\right) + \cutnd\left(W_{G_{Tn^2}},W_{G^*}\right) < \eps,
\]
as was needed.
\end{proof}

\bibliographystyle{plain}
\bibliography{DSG}

\end{document}